\documentclass[10pt]{amsart}
	\usepackage{amsmath, amsfonts, amssymb,mathrsfs,color}
	\usepackage{graphicx}
	\usepackage{tikz-cd}
	\usepackage{tensor}
	\usepackage{xcolor}
	\usepackage{graphicx, xcolor}
	\usepackage{bm}	
	\usepackage{amsthm}	
	\usepackage{amsfonts}	
	\usepackage{amssymb}
	\usepackage{mathrsfs}
	\usepackage{latexsym}
	\usepackage{enumerate}
	\usepackage{dsfont}
	\usepackage{mathtools}
	\usepackage{bbm}
	\usepackage{enumitem}
	\usepackage{subfig}
	\usepackage{comment}
	\usepackage{cancel}
	\usepackage{hyperref}
	\usepackage{blkarray}

	\definecolor{light}{rgb}{0,1,1}

         \newcommand{\Id}{\ensuremath{\text{Id}}}
	\usepackage{todonotes}
	
	\newtheorem{theorem}{Theorem}[section]
	\newtheorem{lemma}[theorem]{Lemma}
	\newtheorem{corollary}[theorem]{Corollary}
	\newtheorem{proposition}[theorem]{Proposition}

	\theoremstyle{definition}
	\newtheorem{definition}[theorem]{Definition}
	\newtheorem{example}[theorem]{Example}

	\theoremstyle{remark}
	\newtheorem{remark}[theorem]{Remark}

	\numberwithin{equation}{section}

	\newcommand{\R}{\mathbb R}  
	\newcommand{\Z}{\mathbb Z}  
	\newcommand{\C}{\ensuremath{\mathbb{C}}}
	\newcommand{\T}{\ensuremath{\mathbb{T}}}

	\newcommand{\fk}{\ensuremath{\mathfrak{k}}}
	
	\newcommand{\fg}{\ensuremath{\mathfrak{g}}}

	\def\stackto #1 { \, {\stackrel{#1}{\longrightarrow}}\, }
	\def\stackTo #1 { \, {\stackrel{#1}{\Longrightarrow}}\, }

	\newcommand{\Lie}{\ensuremath{\text{Lie}}}
	\newcommand{\dolbeault}{\ensuremath{\overline{\partial}}}
	\newcommand{\zbar}{\ensuremath{\overline{z}}}
	\newcommand{\inn}[2]{\ensuremath{\left<{#1},{#2}\right>}}
	\newcommand{\ind}{\text{ind}}

	\newcommand{\Hom}{\ensuremath{\operatorname{Hom}}}

	\renewcommand{\to}{\ensuremath{\nobreak\rightarrow\nobreak}}

	\usepackage{mathabx,epsfig}

	\newcommand{\cq}[1]{\ensuremath{q}}

	\numberwithin{equation}{section}
	\numberwithin{figure}{section}

\begin{document}

\title{Equivariant Index on Toric Contact Manifolds}

\author{Pedram Hekmati}
\address[Pedram Hekmati]{University of Auckland, Department of Mathematics, 1010 Auckland, New Zealand}
\email{p.hekmati@auckland.ac.nz}
\author{Marcos Orseli}
\address[Marcos Orseli]{University of Auckland, Department of Mathematics, 1010 Auckland, New Zealand}
\email{mors019@aucklanduni.ac.nz}
\begin{abstract}     We compute the equivariant index of the twisted horizontal Dolbeault operator on compact toric contact manifolds of Reeb type. The  operator is  elliptic transverse to the Reeb foliation and  its  equivariant index  defines a distribution on the torus. Using the good cone condition, we show that the symbol localises to the closed Reeb orbits corresponding to the edges of the moment cone and obtain an Atiyah-Bott-Lefschetz type formula for the index. For the horizontal Dolbeault operator, we obtain an expression for the index as a sum over the lattice points of the moment cone, by applying an adaptation of the Lawrence-Varchenko polytope decomposition  to rational polyhedral cones. 
\end{abstract}

\maketitle

\section{Introduction}
	There has been considerable interest in $K$-contact and Sasaki manifolds in recent years, in part due to their role in theoretical physics as backgrounds in supersymmetric gauge theories  \cite{festuccia_transversally_2020,qiu_5d_2013}  and the AdS/CFT correspondence \cite{gauntlett_new_2004,gauntlett_sasaki-einstein_2004,martelli_geometric_2006,martelli_geometric_2008}. Let $(M,H)$ be a $(2n+1)$-dimensional compact co-oriented contact manifold with contact form $\alpha$ and associated Reeb vector field $R_\alpha$. We say that $(M,H)$ is  toric if it carries an effective action by a torus $G$ of dimension $n+1$ preserving the contact structure. Recall that $(M,H)$ is of Reeb type if the Reeb vector field is generated by a one-parameter subgroup of $G$. Toric contact manifolds of Reeb type carry an invariant Sasakian structure  \cite{boyer_note_2000} and as shown by Lerman \cite{lerman_contact_2001},  they admit a combinatorial description in terms of their moment map images, which are strictly convex rational polyhedral cones. This structure has been exploited to compute various invariants of toric contact manifolds, such as the volume \cite{goertsches_2017, martelli_geometric_2006}, the first and second homotopy groups \cite{lerman_homotopy_2004}, the equivariant cohomology ring \cite{luo_cohomology_2012} and the cylindrical contact homology \cite{abreu_contact_2012}.

		In this paper, we consider the index of the horizontal Dolbeault operator on compact toric contact manifolds of Reeb type endowed with an invariant Sasakian structure. This operator is the odd dimensional analogue of the Dolbeault operator in Kähler geometry. It appears for instance in \cite{festuccia_transversally_2020,qiu_5d_2013} in the calculation of  perturbative partition functions of certain supersymmetric field theories, in \cite{tievsky_analogues_2008} in relation to deformations of Sasakian structures and in \cite{baraglia_moduli_2016} to compute the dimension of  moduli spaces of instantons on contact 5-manifolds. The operator is  elliptic transverse to the Reeb foliation and on toric Sasaki manifolds, it is elliptic in directions transversal to the $G$-orbits.
		
		In  \cite{atiyah_elliptic_1974}, Atiyah-Singer  proved that  a pseudodifferential operator $A$ that is $G$-transversally elliptic may have an infinite-dimensional kernel and cokernel, but they define a virtual trace-class representation of $G$. The index of $A$ can therefore be defined as a generalised function on $G$ by
			\begin{equation*}
				\ind_G^M(A)(t)=\text{tr}(t|_{\ker A})-\text{tr}(t|_{\text{coker} A}).
			\end{equation*}
			Decomposing $\ker A$ and $\text{coker} A$ into isotypical components, we have
			\begin{equation*}
				\ind_G^M(A)(t)=\sum_{\mu\in \hat{G}} m(\mu)\chi_\mu,
			\end{equation*}
		where $m\colon \hat{G}\to \Z$ encodes the multiplicities of the irreducible $G$-representations appearing in the index character.
		We will compute explicitly the function $m$ when $A$ is the horizontal Dolbeault operator $\dolbeault_H$ and more generally, we derive a localisation theorem for the index when $\dolbeault_H$  is coupled to a holomorphic bundle. 

		Our first main result is an Atiyah-Bott-Lefschetz type formula  for the twisted horizontal Dolbeault operator: 
				\begin{theorem}
				\label{locformulaIndex}
				Let $\dolbeault_H^E$ be the horizontal Dolbeault operator on a compact toric  Sasaki $(2n+1)$-manifold $M$ twisted by a $G$-equivariant transversally holomorphic bundle $E$ over $M$.  For any $t\in G$,
				\begin{equation*}
					\ind^M_G(\dolbeault_H^E)(t)=\sum_{L\in E(C)} \chi_{E|_{L}}(t)\prod_{i=1}^n\left(\frac{1}{1-t^{-{w}^i_L}}\right)^{\pm}\delta(1-t^{\mu_L}),
				\end{equation*}
				where $E(C)$ is the set of edges of the moment cone $C$, $\{w_L^1, \dots, w_L^n\}$ are the isotropy weights  and $\mu_L$ is the weight of the action of $G$ on the closed Reeb orbit corresponding to $L$.
			\end{theorem}
		The signs $\pm$ dictate whether the denominator is expanded about $t=0$ or $t=\infty$ and are fixed by the pairing of the isotropy weights with a polarizing vector, see   Section \ref{index_dolbeault} for detailed explanations.

		 Our method is based on Atiyah's algorithm outlined in \cite{atiyah_elliptic_1974} to stratify $M$ using the torus action and reduce the index calculation to computations on lower dimensional submanifolds. Using the good cone property of toric contact manifolds of Reeb type, we construct a deformation vector field that in fact localises the index to contributions from a finite number of closed Reeb orbits corresponding to the edges of the moment cone.	 
	We note that a general cohomological formula for the index of $G$-transversally elliptic operators was obtained in \cite{berline_chern_1996,berline_indice_1996} and more specifically for contact manifolds in \cite{fitzpatrick_contact_2009}. These formulas are however not well-adapted to computing the multiplicities since they provide an expression for the index that is valid only on a neighbourhood of each $t\in G$. Even in the elliptic case, deducing the function $m$ from the Atiyah-Segal-Singer fixed point theorem is not easy. Our approach is to exploit the combinatorial structure of the manifold $M$ to determine the function $m$ as explicitly as possible. For instance, Theorem  \ref{locformulaIndex}  can be readily applied to compute the dimensions of moduli spaces of  instantons \cite{baraglia_moduli_2016} and transverse Seiberg-Witten monopoles \cite{kordyukov_instantons_2016} on toric Sasaki manifolds.

	We further remark that in Theorem  \ref{locformulaIndex}   the Sasakian structure is not assumed to be quasi-regular and the formula applies in particular to the irregular $Y^{p,q}$ spaces \cite{gauntlett_sasaki-einstein_2004}. When the Sasakian structure is regular, the manifold is a principal circle bundle over a toric K\"ahler manifold $X$. In this case, $\dolbeault_H$  descends to the usual Dolbeault operator on $X$ twisted by the character line bundles of the circle and the index can be computed using the  Lefschetz fixed point formula.

		 Our second result is an expression for the index of $\dolbeault_H$   in terms of the integral points of the moment cone:
		\begin{theorem}
				\label{index_lattice}
		The index of the horizontal Dolbeault operator $\dolbeault_H$ is given by
					\begin{equation*}
						\ind^M_G( \dolbeault_H)(t)=(-1)^n\sum_{\mu\in C^\circ \cap \Z_G^*} t^\mu+\sum_{\mu\in (-C)\cap \Z_G^*} t^\mu,
					\end{equation*}
					where $\Z_G^*$ denotes the dual integral lattice of $\fg$, $C^\circ$ the interior of the moment cone and $-C$  the negative cone.
		\end{theorem} 
	This follows by adapting and applying a version of the Lawrence-Varchenko formula to the polar decomposition of cones over  a polytope (Proposition \ref{LVconesdual2}). Another key ingredient in the proof is Lerman's local description of toric contact manifolds \cite{lerman_contact_2001}, that allows us to relate the weights of the torus action to the moment cone and identify the localisation formula in Theorem \ref{locformulaIndex} with the Lawrence-Varchenko formula for  rational polyhedral cones.		

	In \cite{martelli_geometric_2008}, Martelli-Sparks-Yau considered the Dolbeault operator on  an orbifold resolution of the non-compact K\"ahler cone of a Sasaki manifold and showed that the equivariant index equals the integral points of the moment cone. This corresponds essentially to the second term of the index of $\dolbeault_H$ in Theorem \ref{index_lattice}. A similar limiting argument as in \cite{martelli_geometric_2008} applied to this term would compute the volume of the momentum polytope, which up to a constant equals the volume of the toric Sasaki manifold. Another  related recent work is  by Lin-Loizides-Sjamaar-Song \cite{loizides}, where they study the equivariant index of the basic Dirac operator on Riemannian foliations whose leaf space is symplectic and establish a quantization commutes with reduction theorem. This includes toric $K$-contact manifolds as a special case, however their setup is complementary to ours as their operator only acts on basic sections, corresponding to the invariant part of our index.

		The paper is structured as follows. In Section \ref{toric_contact} we provide a brief review of contact and Sasakian structures, recall some results from \cite{lerman_contact_2001} including a local normal form for toric contact manifolds and Lerman's construction of a toric contact manifold from a good cone. We also introduce the main object of this paper, the horizontal Dolbeault operator. Section \ref{index} introduces the fundamental concepts in the theory of $G$-transversally elliptic operators and our main computational tool, Atiyah's  algorithm for localising the index. In Section \ref{index_dolbeault} we introduce a deformation vector field and apply the localisation argument to derive Theorem \ref{locformulaIndex}. Finally, in Section \ref{lattice}, we prove a cone version of the Lawrence-Varchenko formula and apply it to the index of the horizontal Dolbeault operator to obtain an explicit lattice point formula.

\section{Toric contact manifolds}
	\label{toric_contact}
		Let $M$ be a smooth compact manifold of  dimension $2n+1$.  Recall that a contact structure on $M$ is a hyperplane distribution $H\subset TM$ defined globally by $H=\ker\alpha$ for a  1-form  $\alpha$ such that $d\alpha|_H$ is non-degenerate.  The contact form $\alpha$ defines a volume form $\alpha\wedge (d\alpha)^n$ on $M$ and its conformal class determines a co-orientation of the pair $(M, H)$. 
		Associated to every co-oriented contact manifold is a symplectic cone $(C(M),\omega)$ defined by
		\begin{equation*}
			C(M)=M\times \R\;\;\text{and}\;\;\omega=d(e^r\alpha),
		\end{equation*}
		where $r$ is the coordinate in $\R^+$ and  $2 \frac{\partial}{\partial r}$ is the Liouville vector field.  The Reeb vector field associated to $\alpha$ is the unique vector field $R_\alpha$ satisfying  $\iota_{R_\alpha} \alpha =1$ and $\iota_{R_\alpha} d\alpha=0$.  We let $V \subset TM$ denote the rank one sub-bundle spanned by $R_\alpha$.

	A {\em contact metric structure} on  $(M, \alpha, R_\alpha)$ is a reduction of structure of the tangent bundle to $U(n) \subset GL(2n+1,\mathbb{R})$. Alternatively, it consists of an endomorphism $\Phi \colon TM \to TM$ and a Riemannian metric $g$ such that $\Phi^2 = -I + \alpha \otimes R_\alpha$ and $g( \Phi X , \Phi Y) = g(X,Y) - \alpha(X) \alpha(Y)$ for all vector fields $X,Y$. This yields an orthogonal decomposition $TM = V \oplus H$ together with a unitary structure on $H$. The restriction $J = \Phi|_H$ of $\Phi$ to $H$ defines the complex structure on $H$ and $d\alpha(X,Y) = g(X,\Phi Y)$ restricted to $H$ is the Hermitian $2$-form associated to $J$. 

	We say that $M$ is  a {\em$K$-contact manifold} if $R_\alpha$ is a Killing vector field with respect to $g$. This is equivalent to the characteristic  foliation generated by $R_\alpha$ being a Riemannian foliation. Extending $g$ to the symplectic cone, we obtain a metric $h=dr^2+r^2g$ on $C(M)$ and an associated almost complex structure $J_C$ defined by $h(X,J_CY)=\omega(X,Y)$.  A contact metric structure $(\alpha, R_\alpha, \Phi, g)$  on $M$ is called {\em Sasakian} if  $(h, J_C,\omega)$ is a Kähler structure on $C(M)$.  Sasaki manifolds  constitute the most important class of  $K$-contact manifolds and are the odd dimensional counterparts to Kähler manifolds.

			\begin{example}
				\label{boothby_wang} Geometric quantisation provides  examples of {\em quasi-regular} $K$-contact manifolds, that is when all leaves of the characteristic foliation are circles.
				Let $(B,\omega)$ be a symplectic manifold such that $[\omega]\in H^2(B,\Z)$. Let $M$ denote the  principal $S^1$-bundle over $B$ with  Chern class equal to $[\omega]$. There is a connection form $\alpha$ on $M$ such that $d\alpha=\pi^*\omega$. Since $\omega$ is symplectic, we have 
					$\alpha\wedge (d\alpha)^n=\alpha\wedge\pi^*\omega^n\neq 0,$ so $\alpha$ is a contact form and its Reeb vector field $R_\alpha$ is the generator of the free $S^1$-action on $M$.  Such contact manifolds are regular and the projection $\pi \colon  M \to B$ is known as the Boothby-Wang fibration  \cite{boothby_contact_1958}. 
									This construction generalises to symplectic orbifolds $(B,\omega)$ such that $[\omega] \in H^2(B,\mathbb{R})$ admits a lift to a class $c \in H^2_{{\rm orb}}(B,\mathbb{Z})$, the degree $2$ orbifold cohomology of $B$. Then $c$ defines a Seifert fibration $\pi \colon  M \to B$ carrying a pseudo-free  $S^1$-action and $M$ admits the structure of a quasi-regular $K$-contact manifold. When $B$ is a Kähler orbifold,  $M$ acquires a Sasakian structure.
			\end{example}

				Let $G$ be a torus of dimension $n+1$,  $\fg$  its Lie algebra and $\fg^*$ its dual Lie algebra. We will denote by $\Z_G=\ker(\exp\colon \fg\to G)$ the integral lattice of $\fg$. Suppose that $G$ acts on a manifold $M$. If $v\in \fg$, we  denote by $v(p)\in T_p M$ the tangent vector induced by the action of $G$ on $M$.

			\begin{definition}
				A contact manifold $(M, H)$ of dimension $2n+1$ is called {\em toric} if there is an effective action by an $(n+1)$-dimensional torus $G$ on $M$ preserving the contact form. The {\em  $\alpha$-moment map}  $\phi_\alpha\colon M\to \fg^*$ is defined by 
				\begin{equation*}
				    \inn{\phi_\alpha(p)}{v}=\alpha_p(v(p))
				\end{equation*}
				for all $p\in M$ and $v\in \fg$. The {\em  moment cone} associated with $\phi_\alpha$ is defined by
				\begin{equation*}
				    C=\left\{t\phi_\alpha(p)\ | \  t\geq 0, p\in M  \right\},
				\end{equation*} 
	and can be identified with  the union of $\{0\}$ with the image of the  moment map $e^r \phi_\alpha$ of the lifted Hamiltonian $G$-action on $C(M)$, where $G$ acts trivially on $\R$.
			\end{definition}

	The classification of compact toric contact manifold was completed by Lerman \cite{lerman_contact_2001}. In dimensions greater than three and when the $G$-action is not free, the contact toric manifolds are classified by their moment cones, which are good cones:
			\begin{definition}

				A cone $C\subset \fg^*$ is {\em good} if there exists a minimal set of primitive vectors $v_1,\ldots,v_d\in\Z_G$, with $d\geq n+1$, such that:

				\begin{enumerate}[label=(\roman*)]
				    \item $C=\bigcap\limits_{j=1}^d\{y\in\fg^*\ | \ \left<y,v_j\right>\geq 0\}$,
				    \item Any codimension-$k$ face of $C$, $1\leq k\leq n$, is the intersection of exactly $k$ facets whose set of normals can be completed to an integral base of $\Z_G$.
				\end{enumerate}
			\end{definition}

		\begin{remark}
				Good cones are rational polyhedral, meaning that the normals to the facets are integral vectors.
			\end{remark}

			Toric contact manifolds can be further divided into  Reeb and non-Reeb types.  We say that $M$ is of {\em Reeb type} if $R_\alpha$ is generated by an element $R\in\fg$.

			\begin{theorem}[\cite{boyer_note_2000},  \cite{lerman_contact_2001}]
				\label{strictly_convex}
				If $(M,\alpha)$ is a toric contact manifold of Reeb type, then its moment cone $C$ is a strictly convex good cone. The image of the $\alpha$-moment map $\phi_\alpha$ is a compact convex simple polytope $P$ given by the intersection of the characteristic hyperplane
				\begin{equation*}
				   \mathcal H=\left\{\eta\in \fg^*\ | \ \eta(R)=1 \right\},
				\end{equation*}
				determined by the vector $R$, with the moment cone $C$.
			\end{theorem}

			\begin{remark}
				Strictly convex means that $C$ contains no linear subspaces of positive dimension, so it is a cone over a polytope. Toric contact manifolds with good moment cones $C$ that are not strictly convex are diffeomorphic to $\T^k\times S^{k+2l-1}$, for some $k>1,l\geq 0$  \cite{lerman_completely_2002}.  
			\end{remark}

	Toric contact manifolds with an invariant $K$-contact structure must be of Reeb type \cite{lerman_homotopy_2004} and they always admit  an invariant Sasakian structure \cite{boyer_note_2000}. In the sequel, we will therefore assume that our toric contact manifolds are of Reeb type of dimension greater than three and equipped with an invariant Sasakian structure. We will need the following result  characterising the Reeb vector fields associated to a Sasakian structure:

			\begin{theorem}[\cite{martelli_geometric_2006}]
				\label{sasaki_reebs}
				Let $v_1,\ldots,v_d\in\fg$ be the defining integral normals of the moment cone $C\in\fg^*$ associated with a toric contact manifold of Reeb type $(M,H)$. A vector $R\in \fg$ generates the Reeb vector field of an invariant Sasakian $1$-form $\alpha$ such that $H=\ker\alpha$ if and only if
				\begin{equation*}
				    R=\sum_{j=1}^d a_jv_j,\text{ with } a_j\in \R^+ \text{ for all } j=1,\ldots,d.
				\end{equation*}
			\end{theorem}

			\begin{example}
				\label{boothby_wang_toric} Returning to Example \ref{boothby_wang}, if $(B,\omega)$ is an integral toric symplectic manifold, then the  principal $S^1$-bundle over $B$ is a good toric contact manifold with moment cone
				$$C= \{ r (x,1)\in \R^n\times \R \ | \  x\in P, r\geq 0\}, $$
				where $P\subset \R^n$ is the integral Delzant polytope associated with $B$.
			\end{example}

	   \begin{example} For each pair of coprime integers $p,q$ with $0<q<p$, the $Y^{p,q}$ spaces  are  toric Sasaki-Einstein metrics on $S^2 \times S^3$. The Sasakian structure is   irregular of rank $2$ whenever $4p^2 - 3q^2$ is a not perfect square  \cite{gauntlett_sasaki-einstein_2004}. In higher dimensions they generalise  to the family of toric contact manifolds $N^{2n+1}_{k,m}$, $n\geq 2$, $k\geq 1$ and $0\leq m<kn$, associated to the good cones $C(k,m)\subset (\R^{n+1})^*$ defined by the normals
		    	\begin{align*}
		    	 	v_i&=e_i+e_{n+1},  \
		    	 	v_n=-\sum_{i=0}^{n-1}e_i+me_n+e_{n+1}, \  
		    	 	v_-=ke_n+e_{n+1},  \
		    	 	v_+=-e_n+e_{n+1},
		    	\end{align*}
		    	where $e_i\in (\R^{n+1})$, $i=1,\ldots,n$, are the canonical basis vectors of $\R^{n+1}$. Unlike the $Y^{p,q}$ spaces, $N^{2n+1}_{k,m}$ are not all diffeomorphic \cite{abreu_kahlersasaki_2010}.
		    \end{example}

		\subsection{Lerman's construction}
			\label{lermansconstruction}
					The classification of toric contact manifolds of Reeb type is analogous to Delzant's classification of toric symplectic manifolds \cite{boyer_note_2000}, \cite{lerman_contact_2001}. In this section, we briefly recall the construction of a toric contact manifold from its moment cone and elucidate the relation with the isotropy weights.

			Let $C\subset \fg^*$ be a strictly convex good cone given by
			\begin{equation*}
				C=\bigcap_{i=1}^d\{\eta\in \fg^*\ | \  \eta(v_i)\geq 0\},
			\end{equation*}
			where $v_i\in \Z_G$, $i=1,\ldots, d$, are the inward pointing normals of $C$ and $\dim \fg^*>2$.

			Let $\{e_1,\ldots,e_d\}$ denote the standard basis of $\R^d$ and define the map $\beta\colon \R^d\to \fg$ by 
			  	$\beta(e_i)=v_i.$
			Denote by $\fk$ the kernel of $\beta$. Since $C$ is strictly convex, $\beta$ is surjective and we have the short exact sequences
			\begin{equation*}
				0\to \fk \xrightarrow{\iota} \R^d\xrightarrow{\beta} \fg \to 0 \;\text{ and }\; 0\to \fg^* \xrightarrow{\beta^*} (\R^d)^*\xrightarrow{\iota^*} \fk^*  \to 0.
			\end{equation*}
			Since $\beta(\Z^d)\subset \Z_G$, $\beta$ induces a map $\tilde{\beta}\colon \T^d=\R^d/\Z^d\to \fg/\Z_G=G$. Let
			\begin{equation*}
				K=\left\{[t]\in \T^d\ | \  \sum_{i=1}^d t_iv_i\in \Z_G\right\}
			\end{equation*}
			denote the kernel of $\tilde{\beta}$. It is a compact abelian subgroup with Lie algebra $\fk=\ker(\beta)$.
			Consider the standard action of $\T^d$ on $(\C^d, \omega_{st}=i/{2\pi} \sum_{j=1}^d dz_j\wedge d\zbar_j)$ given by
			\begin{equation*}
				[t]\cdot(z_1,\ldots,z_d)=(e^{2\pi i t_1}z_1,\ldots,e^{2\pi i t_d}z_d).
			\end{equation*}
			The corresponding moment map $\phi\colon \C^d\to (\R^d)^*$ is given by
			\begin{equation*}
				\phi(z_1,\ldots,z_d)=\sum_{j=1}^d\mid z_j \mid^2 e_j^*,
			\end{equation*}
					where $\{e_j^*\}$ is the basis dual to the canonical basis $\{e_j\}$.
			Since $K$ is a subgroup of $\T^d$, it acts on $\C^d$ with moment map
			\begin{equation*}
				\phi_K(z_1,\ldots,z_d)=\sum_{j=1}^d\mid z_j \mid^2 \iota^*(e_j^*)\in \fk^*.
			\end{equation*}
			The reduced space  $W_C=\frac{(\phi_K^{-1}(0)\setminus \{0\})}{K}$ is a toric symplectic cone with symplectic form $\omega_C$ induced by $\omega_{st}$. It carries an action of $G=\T^d/K$ induced by the $\T^d$-action and an action of $\R$ induced by the standard radial $\R$-action on $\C^d$. 

			Let $\sigma$ be a section of $\beta\colon \R^d\to \fg$, giving a splitting  $\R^d \cong \iota(\fk)\oplus \sigma(\fg)$. Since $\sigma$ is injective, its image defines an $n$-torus $\sigma(G)\subset \T^d$. The action of $G$ on $(\phi_K^{-1}(0)\setminus \{0\})$ via $\sigma(G)\subset \T^d$ is Hamiltonian with moment map $\tilde{\phi}=\sigma^*\circ\phi\colon (\phi_K^{-1}(0)\setminus \{0\})\to \fg^*$. The $G$-action and the moment map $\tilde{\phi}$ descend to the quotient $W_C$ making it a Hamiltonian $G$-space with moment map
			\begin{align*}
				\phi_G\colon W_C&\to \fg^*\\
				[z_1,\ldots,z_d]&\mapsto\sigma^*(\phi(z_1,\ldots,z_d)),
			\end{align*}
			where we denote by $[z_1,\ldots,z_d]\in W_C$ the class of $(z_1,\ldots,z_d)\in (\phi_K^{-1}(0)\setminus \{0\})$ in the quotient. The image of $\phi_G$ is the cone $C\setminus\{0\}$.
			The sphere $S^{2d-1}=\{z\in \C^d; \lvert z \rvert\}$ is a $\T^d$-invariant hypersurface of contact type in $\C^d$ and
			\begin{equation*}
				M_C=\frac{(\phi_K^{-1}(0)\bigcap S^{2d-1})}{K}
			\end{equation*}
			is a $G$-invariant hypersurface of contact type in $W_C$. Therefore it has a toric contact structure induced by the $G$-invariant contact form $\alpha=i_{X_C}\omega_C$, where $X_C$ is the Liouville vector field induced by the $\R$-action on $W_C$. The moment cone of $(M_C,\alpha)$ is $C$.

			\begin{lemma}
			\label{preimagecone}
				$\phi_K^{-1}(0)=\phi^{-1}(\beta^*(C))$
			\end{lemma}
			\begin{proof}
				Since $0\to \fg^* \xrightarrow{\beta^*} (\R^d)^*\xrightarrow{\iota^*} \fk^*  \to 0$ is exact, we have $(\iota^*)^{-1}(0)=\beta^*(\fg^*)$. Therefore $\phi_K^{-1}(0)=(\iota^*\circ\phi)^{-1}(0)=\phi^{-1}((\iota^*)^{-1}(0))=\phi^{-1}(\beta^*(\fg^*))=\phi^{-1}(\beta^*(\fg^*)\cap \phi(\C^d))$.
			It follows from
				\begin{align*}
				    \beta^*(\fg^*)\cap \phi(\C^d)&=\{\beta^*(\eta) \ | \ \eta\in\fg^* and \left<\beta^*(\eta),e_i\right>\geq 0 \text{ for all }i\}\\
				    &=\{\beta^*(\eta)\ | \ \eta\in\fg^* and \left<\eta,\beta(e_i)\right>\geq 0 \text{ for all }i\}\\
				    &=\{\beta^*(\eta)\ | \ \eta\in\fg^* and \left<\eta,v_i\right>\geq 0 \text{ for all }i\}\\
				    &=\{\beta^*(\eta)\ | \ \eta\in C\}
				\end{align*}
				that $\phi_K^{-1}(0)=\phi^{-1}(\beta^*(C))$.
			\end{proof}

			The following lemma informs us how to read the isotropy groups from the moment cone.

			\begin{lemma}
				\label{isotropy_toric_contact}
				Let $(M,\alpha)$ be a toric contact manifold of Reeb type with moment cone $C$. Let $p\in M$ and $\eta=\phi_\alpha(p)$ the image of $p$ under the $\alpha$-moment map. If $\eta(v_i)=0$, for a subset of indices $i\in I\subset \{1,\ldots,d\}$, then its isotropy Lie algebra $\fg_p$ is generated by the vectors $v_i$, $i\in I$. 
			\end{lemma}
			\begin{proof}
				Lerman's construction implies that $M\cong M_C$ and that $C=\phi_G(W_C)$. Therefore $\eta(v_i)=0$ if and only if $\phi_G(p)(v_i)=0$, where we are considering $p$ as an element of $W_C$ via the inclusion $M\cong M_C\subset W_C$. From Lemma \ref{preimagecone} we have that $\phi_K^{-1}(0)=\phi^{-1}(\beta^*(C))$.
				Let $z=(z_1,\ldots,z_d)\in \phi^{-1}_K(0)$ be such that $\beta^*(\eta)=\phi(z)$. Then
				\begin{equation*}
				    \left|z_j\right|^2=\left<\phi(z),e_j\right>=\left<\beta^*(\eta),e_j\right>=\left<\eta,\beta(e_j)\right>=\left<\eta,v_j\right>.
				\end{equation*}
				Therefore $z_j=0$ if and only if  $\eta(v_j)=0$.
				The torus $\T^d$ acts on $\phi^{-1}_K(0)\subset \C^d$ via the standard action and $\eta_(v_j)=0$ if and only if  $e_j\in \mathfrak{t}^d_z$. The $G$-action on $\phi_K^{-1}(0)$ is given by a section $\sigma$ of $\beta\colon \T^d\to G$. Since $K=\ker \beta$, we have $\sigma_*(v_j)=e_j+k$, where $k\in \fk$ and $\sigma_*$ is the Lie algebra map induced by $\sigma$. It follows that $\fg_p=\mathfrak{t}^d_z/\fk$, therefore $[\sigma_*(v_j)]=[e_j]\in \mathfrak{t}^d_z/\fk$. Since the $G$-action is given by the section $\sigma$, we have that $v_j\in \fg_p$ is equivalent to $[\sigma_*(v_j)]\in \mathfrak{t}^d_z/\fk$, and therefore $v_j\in \fg_p$ if and only if $\eta(v_j)=0$.
				Since $C$ is a good cone, the $v_j$ satisfying $\eta(v_j)=0$ form an integral basis of $\fg_p$.
			\end{proof}

			Next we consider the weights of the isotropy representations. First we need a lemma from \cite{delzant_hamiltoniens_1988}:
			\begin{lemma}
				\label{symplectic_weights_basis}
				Let $\rho\colon G\to GL(V)$ be a faithful representation of a torus $G$ preserving a symplectic form $\omega$ such that $\dim V=2\dim G$.
				If $\rho$ is faithful, then its weights form a basis of the weight lattice $\Z^*_G$ of $G$.
			\end{lemma}

	We have the following specialisation of Lerman's local normal  form for the moment map \cite{lerman_contact_2001}, when restricted to a vertex.

			\begin{theorem}
				\label{equivariant_darboux_contact}
				Let $p\in M$ be such that $\phi_\alpha(p)$ is a vertex of the convex polytope $P=\phi_\alpha(M)$ and let $V=H_p$ be the fibre of the contact distribution on $p$. The isotropy group $G_p$ acts on $V$ preserving the symplectic form $d\alpha|_{V}$. Then
				\begin{equation*}
					\fg_p^\circ=\R\phi_\alpha(p) 	
				\end{equation*} 
				and we can choose a splitting
				\begin{equation*}
					\fg^*=\fg_p^\circ\oplus\fg_p^*=\R\phi_\alpha(p)\oplus \fg_p^*.
				\end{equation*}
				Let $i\colon \fg_p^*\to \fg^*$ be the corresponding embedding. Then there exists a $G$-invariant neighbourhood $U$ of the zero section $G\cdot[1,0]$ in
				\begin{equation*}
					N=G\times_{G_p} V
				\end{equation*}
				and an open $G$-equivariant embedding $\varphi\colon U\to M$ with $\varphi([1,0])=p$ and a $G$-invariant $1$-form $\alpha_N$ on $N$ such that
				\begin{enumerate}
					\item $\varphi^*\alpha=e^f\alpha_N$ form some function $f\in C^\infty(U)$;
					\item the $\alpha_N$-moment map $\phi_{\alpha_N}$ is given by
					\begin{equation*}
						\phi_{\alpha_N}([a,v])=\phi_\alpha(p)+i(\phi_V(v)),
					\end{equation*}
					where $\phi_V\colon V\to \fg^*$ is the moment map for the representation of $G_p$ on $V$.
				\end{enumerate}

				Consequently,
				\begin{equation*}
					\phi_\alpha\circ \varphi([a,v])=(e^f\phi_{\alpha_N})([a,v])=e^{f([a,v])}(\phi_\alpha(p)+i(\phi_V(v)))
				\end{equation*}
				for some $G$-invariant function $f$ on $N$.
			\end{theorem}

			The following corollary shows how the isotropy weights relate to the moment cone.
			\begin{corollary}
				\label{local_toric_contact}
				Let $p\in M$ be such that $\phi_\alpha(p)$ is a vertex of the convex polytope $P=\phi_\alpha(M)$. Then the representation $G_p\to GL(V)$ is faithful, and the weights of the action of the isotropy group $G_p$ on $V=H_p$, the fibre of the contact distribution at $p$, form a basis of the weight lattice $\Z_{G_p}^*$ of $\fg_p^*$ that is dual to the basis $\{v_1^p,\ldots,v_n^p\}$ of $\fg_p$, where $v_1^p,\ldots,v_n^p$ are the normals to the faces of the moment cone $C(\phi_\alpha)$ meeting at $\phi_\alpha(p)$.
			\end{corollary}

			\begin{proof}
				Theorem \ref{equivariant_darboux_contact} ensures that there is a neighbourhood of $G\cdot p$ that is $G$-equivariantly diffeomorphic to a neighbourhood of the zero section of $N=G\times_{G_p} V$. Since the action of $G$ on $N$ is effective, the representation of $G_p$ on $V$ must be faithful.
				The image of the moment map $\phi_V$ is
				\begin{equation*}
					\phi_V(V)=\left\{\sum_{i=1}^ns_iw^i_p \ | \ s_i\geq 0\right\}\subset \fg_p^*,
				\end{equation*}
				where the $w^i_p$ are the weights of the isotropy action of $G_p$ on $V=H_p$. Therefore the weights $w^i_p$ generate the edges of the cone $\phi_V(V)$. 
				Lemma \ref{symplectic_weights_basis} shows that the weights $w^i_p$ form a basis of the integral lattice of $\fg^*_p$. It follows that the weights $i(w^j_p)\in \fg^*$ satisfy $i(w^j_p)(v_k^p)=\delta_{jk}$. This implies that $w^j_p(v_k^p)=\delta_{jk}$, where we are viewing the normals $v_1^p,\ldots,v_n^p$ as elements of $\fg_p$. Thus $w^1_p,\ldots,w^n_p\in \fg^*_p$ is the dual basis to $v_1^p,\ldots,v_n^p$.
			\end{proof}

			\begin{remark}
				Corollary \ref{local_toric_contact} allows us to read the weights of the isotropy representation $G_p\to GL(V)$ from the moment cone; one simply needs to choose a vector $v_0^p$ completing the set $\{v_1^p,\ldots,v_n^p\}$ to an integral basis of the lattice $\Z_G\subset \fg$.
			\end{remark}
			
	\subsection{The horizontal Dolbeault operator}\label{dolbeault}
		Let $M$ be a toric contact manifold of Reeb type with $\dim M > 3$ endowed with a Sasakian structure  $(\alpha, R_\alpha, \Phi, g)$. The transverse complex structure $J=\Phi|_H$ allows us to introduce a horizontal Dolbeault operator $\dolbeault_H$, and more generally $\dolbeault_H^E$ twisted by a transverse holomorphic bundle $E$. We briefly recall the definitions.

			Let $\Omega^k_H(M) = \{ \omega \in \Omega^k(M) \ | \ \iota_{R_\alpha} \omega = 0 \} = \Gamma(M, \bigwedge^k H^*)$ denote the space of {\em horizontal} $k$-forms. The projection operators
	    	$P_V=\alpha\wedge\iota_{R_\alpha}$ and  $P_H=1-P_V$
	    determine a splitting 
	    \begin{equation*}
	    	\Omega^*(M)=\Omega^*_V(M)\oplus\Omega_H^*(M)=P_V(\Omega^*(M))\oplus P_H(\Omega^*(M))
	    \end{equation*}
	    into horizontal and vertical forms, and $d_H =P_H\circ d$ is a differential on $\Omega^k_H(M)$. The transverse complex structure gives the usual decomposition of $\Omega^*_H(M)\otimes \C$ into horizontal $(p,q)$-forms and defines the horizontal Dolbeault complex:

	    \begin{equation*}
	        0\to\Omega^{0,0}_H(M)\xrightarrow{\dolbeault_H}\Omega^{0,1}_H(M)\xrightarrow{\dolbeault_H}\cdots\xrightarrow{\dolbeault_H}\Omega_H^{0,n}(M)\to 0,
	    \end{equation*}
	with the associated symbol complex 
	    \begin{equation*}
	    	0\to\pi^*(\textstyle{\bigwedge}^{\!0,0} H^*)\xrightarrow{\sigma(\dolbeault_H)}\pi^*(\textstyle{\bigwedge}^{\!0,1} H^*)\xrightarrow{\sigma(\dolbeault_H)}\cdots\xrightarrow{\sigma(\dolbeault_H)}\pi^*(\textstyle{\bigwedge}^{\!0,n} H^*)\to 0,
	    \end{equation*}
	    where $\pi\colon T^*M\to M$ is the projection and
	    $		\sigma(\dolbeault_H)_{(p,\xi)}(\bar{w})=\xi_H^{\!0,1}\wedge \bar{w},$
	    	where $(p,\xi)\in T^*M$, $\bar{w}\in \textstyle{\bigwedge}^{\!0,k} H^*|_p$ and $\xi_H^{\!0,1}=P_H(\xi)^{0,1}$ is the $(0,1)$-component of the horizontal projection of $\xi$.

	    	A $G$-equivariant vector bundle $\pi_E\colon E\to M$ is transversally holomorphic if with respect to an open cover $\left\{U_\alpha\right\}$ of $M$, it is determined  by transition functions $g_{\alpha\beta}$ satisfying $\dolbeault_H(g_{\alpha\beta})=0.$ The twisted horizontal Dolbeault operator $\dolbeault_H^E$ is   given locally by
	    	\begin{align*}
	    		\dolbeault_H^E\colon \Omega^{0,k}_H(M,E)&\to \Omega^{0,k+1}_H(M,E)\\
	    		w\otimes u_\alpha&\mapsto \dolbeault_H(w)\otimes u_\alpha.
	    	\end{align*}
	and we have the twisted horizontal Dolbeault complex
	    \begin{equation*}
	    	0\to\Omega^{0,0}_H(M,E)\xrightarrow{\dolbeault_H^E}\Omega^{0,1}_H(M,E)\xrightarrow{\dolbeault_H^E}\cdots\xrightarrow{\dolbeault_H^E}\Omega_H^{0,n}(M,E)\to 0,
		\end{equation*}
	with the symbol 
	 $   		\sigma(\dolbeault_H^E)_{(p,\xi)}(\bar{w}\otimes \bar{e})=\xi_H^{\!0,1}\wedge \bar{w}\otimes\bar{e},
	  $
	    	where  $\bar{w}\otimes\bar{e}\in \textstyle{\bigwedge}^{\!0,k} H|_p\otimes E|_p$.

\section{Equivariant Index}
	\label{index}
	\subsection{Transversally elliptic operators}
		In this section, we collect  relevant facts about transversally elliptic  operators \cite{atiyah_elliptic_1974,paradan_witten_2019}.

		Let $G$ be a compact Lie group and $M$ a compact $G$-manifold. We  denote by $\pi\colon T^*M\to M$ the natural projection. A complex of $G$-invariant pseudodifferential operators  is called transversally elliptic, if it is elliptic in the directions transversal to $G$-orbits.  	More precisely, let $T_G^* M \subset T^*M$ be the closed subset defined by the union of conormals to the $G$-orbits, 
		\begin{equation*}
			T_G^*M=\{(p,\xi)\in T^*M \ | \ \xi(v(p))=0\text{ for all }v\in \fg\}.
		\end{equation*}
			A complex of $G$-invariant pseudodifferential operators $P$ is a sequence
			\begin{equation*}
				0\to \Gamma(E^0)\xrightarrow{P_1}\Gamma(E^1)\xrightarrow{P_2}\cdots\xrightarrow{P_{n}}\Gamma(E^n)\to 0,
			\end{equation*}
			where the $P_i$ are $G$-invariant pseudodifferential operators and $\Gamma(E^i)$ is the  space of sections of the $G$-vector bundle $E^i$, $i=1,\ldots,n$.
	Its symbol  $\sigma_P$ on $T^*M$ is given by the complex
				\begin{equation*}
					0\to \pi^*E^0\xrightarrow{\sigma_1}\pi^*E^1\xrightarrow{\sigma_2}\cdots\xrightarrow{\sigma_{n}}\pi^*E^n\to 0,
				\end{equation*}
				where $\sigma_i=\sigma(P_i)$ is the symbol of the pseudodifferential operator $P_i$. Let  $$\text{Char}(\sigma_P)=\{(p,\xi)\in T^*M \ | \ \sigma_P(p,\xi) \text{ is not exact}\}$$ denote the characteristic set of the complex.
			\begin{definition}
				A complex of $G$-invariant pseudodifferential operators $P$ is {\em $G$-transversally elliptic} if  $\text{Char}(\sigma_P)\bigcap T^*_G M$ is compact.
			\end{definition}

			A $G$-transversally elliptic symbol $\sigma_P$ defines an element $[\sigma_P]\in K_G^*(T^*_G M)$. Conversely, given a symbol class $[\sigma]\in K_G^*(T^*_G M)$, there is a $G$-transversally elliptic pseudodifferential operator $A\colon \Gamma(M,E)\to \Gamma(M,F)$ such that $\sigma(A)=\sigma$. 
			 Let $\hat{G}$ be the set of isomorphism classes of irreducible complex representations of $G$ and $\hat{R}(G)$ be the space of $\Z$-valued functions on $\hat{G}$.  The elements $V\in \hat{R}(G)$ are thus infinite series
				\begin{equation*}
					V=\sum_{\mu\in \hat{G}} m(\mu)\chi_{V_\mu}
				\end{equation*}
				with $m(\mu)\in\Z$.		
			The kernel of $A$ is not finite-dimensional, but it is shown in \cite{atiyah_elliptic_1974} that for every $\mu\in \hat{G}$ the space $\Hom_G(V_{\mu},\ker(A))$ is a finite dimensional vector space of dimension $m(V_\mu, A)$. The integer $m(V_\mu, A)-m(V_\mu, A^*)$ depends only on the class of the symbol $\sigma(A)$ in $K_G^*(T^*_G M)$ and the index of $\sigma$ is defined by
		    \begin{equation*}
		        \ind^M_G(\sigma)=\sum_{\mu\in \hat{G}} (m(V_\mu, A)-m(V_\mu, A^*))\chi_{V_\mu},
		    \end{equation*}
	where the adjoint $A^*$ of $A$ is also a $G$-transversally elliptic pseudo-differential operator and defined by choosing a $G$-invariant metric.
		    Atiyah showed in \cite{atiyah_elliptic_1974} that $\ind_G(\sigma)$ defines a distribution on $G$. The index depends only on the symbol class $[\sigma]\in K_G^*(T^*_G M)$ and descends to a map
		    \begin{equation*}
		        \ind^M_G\colon K_G^*(T^*_G M)\to \hat{R}(G).
		    \end{equation*}

	The following basic example will be important in the sequel.
			\begin{example}
				\label{zero_operator}
				Let $M=S^1$ with $S^1$ acting on $M$ by left translation. Let $E=M\times \C$ be the trivial line bundle and $F=M\times \{0\}$. Then $\Gamma(M,E)=C^\infty(S^1)$ and $\Gamma(M,F)=0$.
				The operator $D=0\colon \Gamma(M,E)\to \Gamma(M,F)$ has symbol
				\begin{align*}
					\sigma_D\colon \pi^*E&\to \pi^*F\\
					(\xi,e)&\mapsto (\xi,0).
				\end{align*}
				Since $T^*_{S^1} M= M\times \{0\}$, the operator $D$ is $S^1$-transversally elliptic.
				The index of $\sigma_D$ is
				\begin{equation*}
					\ind^M_{S^1}(\sigma_D)(t)=\sum_{n\in\Z} \chi_{V_n}(t)=\sum_{n\in \Z} t^n,
				\end{equation*}
				where $V_n=\C$ is the $S^1$-module with the action $t\cdot z=t^nz$.
			\end{example}
	Let $M$ be a $(2n+1)$-dimensional toric contact manifold of Reeb type, $n>1$,  with an invariant Sasakian structure. Let $\dolbeault_H^E$ be the twisted horizontal Dolbeault complex on $M$ as defined in Section \ref{dolbeault}.  
			\begin{proposition}
				\label{dolbeault_is_transv}
			 $\sigma(\dolbeault_H^E)$ is a $G$-transversally elliptic symbol.
			\end{proposition}
			\begin{proof}	
				Recall that the symbol $\sigma_k(\dolbeault_H^E)\colon \pi^*(\textstyle{\bigwedge}^{\!0,k} H^*\otimes E)\to \pi^*(\textstyle{\bigwedge}^{\!0,k+1} H^*\otimes E)$ is given by
				\begin{equation*}
			    	\sigma_k(\dolbeault_H^E)_{(p,\xi)}(\bar{w}\otimes \bar{e})=\xi_H^{\!0,1}\wedge \bar{w}\otimes\bar{e},
			    \end{equation*}
			    where $(p,\xi)\in T^*M$, $\bar{w}\otimes\bar{e}\in \textstyle{\bigwedge}^{\!0,k} H|_p\otimes E|_p$. The complex $\sigma(\dolbeault_H^E)_{(p,\xi)}$ is exact so long as $\xi_H^{\!0,1}\neq 0$, that is when $\xi$ is not a multiple of the contact form $\alpha$.
			    Since $M$ is of Reeb type, a vector $R\in\fg$ generates the Reeb vector field. We have
			    \begin{equation*}
			    	(p,\xi)\in T^*_G M \implies \xi(v(p))=0\text{ for all $v\in \fg$}.
			    \end{equation*}
			    Since the action of $G$ generates the Reeb vector field, we have $\xi(R(p))=0$. By $\alpha_p(R(p))=1$, it follows that $\xi$ is not a multiple of $\alpha_p$ and therefore $\sigma(\dolbeault_H^E)_{(p,\xi)}$ is exact.
						
			\end{proof}

	We conclude by recalling the multiplicative and excision properties of the index \cite{atiyah_elliptic_1974}. 
			Consider a compact Lie group $G_2$ acting on two manifolds $M_1$ and $M_2$ and assume that another compact Lie group $G_1$ acts on $M_1$ commuting with the action of $G_2$. The exterior product of vector bundles induces a multiplication map
			\begin{equation}
				\label{kmultiplication}
			 	\boxtimes\colon K_{G_1\times G_2}(T^*_{G_1} M_1)\otimes K_{G_2}(T^*_{G_2} M_2)\to K_{G_1\times G_2}(T^*_{G_1\times G_2} (M_1\times M_2)).
			 \end{equation}
	 
			 \begin{theorem}[Multiplicative Property]
			 	 \label{multiplicative_property}
			 	For any $\sigma_1\in K_{G_1\times G_2}(T^*_{G_1} M_1)$ and any $\sigma_2\in K_{G_2}(T^*_{G_2} M_2)$, we have
			 	\begin{equation*}
			 		 		\ind_{G_1\times G_2}^{M_1\times M_2}(\sigma_1\boxtimes\sigma_2)=\ind_{G_1\times G_2}^{M_1}(\sigma_1)\ind_{G_2}^{M_2}(\sigma_2).
			 	\end{equation*}
			 \end{theorem}

			 \begin{theorem}[Excision Property]
			 	Let $j\colon U\to M$ be an open $G$-embedding into a compact $G$-manifold $M$. We have a pushforward map $j_*\colon K_G(T_G^* U)\to K_G(T_G^*M)$ and the composition
			 	\begin{equation*}
			 		K_G(T_G^* U)\xrightarrow{j_*} K_G(T_G^*M)\xrightarrow{\ind^M_G} \hat{R}(G)
			 	\end{equation*}
			 	is independent of $j\colon U\to M$.
			 \end{theorem}

			 The product of a symbol $\sigma$ by a $G$-equivariant vector bundle $E$ is the symbol given by
			\begin{equation*}
				(\sigma \otimes E)(p,\xi)=\sigma(p,\xi)\otimes \Id_E.
			\end{equation*}
			Note that the symbol $\sigma(\dolbeault_H^E)$ is of this form. 
			 
			\begin{proposition}
			\label{trivial_twist}
				Let $\sigma\in K_G(T^*_G M)$, $E$ a $G$-module and $\underline{E}$  the corresponding trivial $G$-equivariant bundle over $M$, then
				\begin{equation*}
					\ind^M_G(\sigma\otimes \underline{E})=\ind^M_G(\sigma)\otimes \chi_E\in \hat{R}(G),
				\end{equation*}
				where $\chi_E$ is the character of the $G$-module $E$.
			\end{proposition}
		\subsection{Localisation}
			\label{localisation_section}
			In this section, we  review the $K$-theoretic localisation method  for computing the index of transversally elliptic operators developed  in  \cite{atiyah_elliptic_1974}. The core idea is to choose a filtration of the manifold that allows one to decompose the symbol into contributions from lower dimensional spaces, essentially reducing the problem to  a computation on vector spaces. The main ingredient involved in the computations is Atiyah's pushed symbol $\sigma^\epsilon$. 

		    Let $G$ be a $n$-dimensional torus acting on $\C^n$ with no fixed vector and weights $w^i\in \fg^*$, $i=1,\ldots,n$. A $G$-invariant Riemannian metric $h$ on $\C^n$ induces  an isomorphism
		    \begin{align*}
				T\C^n&\to T^*\C^n\\
				v&\mapsto\tilde{v}=h(v,\cdot).
			\end{align*}
			Given $\epsilon\in\fg$ we will denote by $\epsilon(p)\in T_p \C^n$ the vector generated by $\epsilon$ at $p$ and by $\widetilde{\epsilon(p)}\in T^*_p \C^n$ its image under the isomorphism defined by $h$. 
		    Let $\sigma(\dolbeault)$ be the symbol
		    \begin{equation*}
		        0\to \pi^*(\textstyle{\bigwedge}^{\!0,0} T^*\C^n)\xrightarrow{\sigma(\dolbeault)}\pi^*(\textstyle{\bigwedge}^{\!0,1} T^*\C^n)\xrightarrow{\sigma(\dolbeault)}\cdots\xrightarrow{\sigma(\dolbeault)}\pi^*(\textstyle{\bigwedge}^{\!0,n} T^*\C^n)\to 0,
		    \end{equation*}
		    where $\pi\colon T^*\C\to \C$ is the projection and $\sigma_{(p,\xi)}(\dolbeault)(w)=\xi^{\!0,1}\wedge w$.
		    The symbol $\sigma(\dolbeault)$ is exact away from the zero section of $T^*\C^n$, in fact, $\text{Char}(\sigma(\dolbeault))=\C\times \{0\}$. Since $\text{Char}(\sigma(\dolbeault))\cap T^*_G\C^n=\text{Char}(\sigma(\dolbeault))$ is non-compact, it is not a $G$-transversally elliptic symbol. Atiyah shows in \cite{atiyah_elliptic_1974} how to obtain a $G$-transversally elliptic symbol $\sigma^\epsilon$ by deforming $\sigma(\dolbeault)$ using the $G$-action. Namely, let $H_i=\left\{\epsilon\in\fg \ | \ w^i(\epsilon)=0\right\}$ be the hyperplane in $\fg$ determined by the weight $w^i$ and pick a vector $\epsilon\in\fg$ away from the hyperplanes $H_i$.

			\begin{definition}
				Atiyah's  pushed symbol $\sigma^\epsilon$ is defined by
		    \begin{equation*}
		        \sigma^\epsilon_{(p,\xi)}(w)=\sigma(\dolbeault)_{(p,\xi+g(\mid\xi\mid )\widetilde{\epsilon(p))}}(w)=(\xi+g(\mid\xi\mid )\widetilde{\epsilon(p)})^{\!0,1}\wedge w,
		    \end{equation*}
		    where $\epsilon(p)\in T_p \C^n$ is the tangent vector generated by $\epsilon$ at $p$, $(p,\xi)\in T^*\C^n$, $w\in \textstyle{\bigwedge}^{\!0,k} T^*\C^n$ and $g$ is a bump function supported in a small neighbourhood of $0$, so that $\sigma^\epsilon$ coincides with $\sigma$ away from the zero section of $T^*\C^n$. 
		    \end{definition}

			Since $\sigma$ is an isomorphism away from the zero section of $T^*\C^n$, we know that $\sigma^\epsilon(p,\xi)$ is not isomorphism if and only if $\xi+g(\mid\xi\mid )\widetilde{\epsilon(p)}=0$. Therefore, 
			\begin{equation*}
			    \text{Char}(\sigma^\epsilon)=\{(p,\xi)\in T^*\C^n \ | \ \xi+g(\mid\xi\mid )\widetilde{\epsilon(p)}=0\}
			\end{equation*}
			is still non-compact. If $(p,\xi)\in T^*_G \C^n|_p$, let $t\in \R_+$ be such that $g(\mid t\xi\mid)>0$. Then 
			\begin{align*}
				t\xi+g(\mid t\xi\mid )\widetilde{\epsilon(p)}=0&\implies t\xi(\epsilon(p))+g(\mid t\xi\mid )\widetilde{\epsilon(p)}(\epsilon(p))=0\\
				&\implies g(\mid t\xi\mid )h(\epsilon(p),\epsilon(p))=0\\
				&\implies \epsilon(p)=0\implies \widetilde{\epsilon(p)}=0\\
				&\implies \xi=0\text{ and } p=0,
			\end{align*}
			where we have used the facts that $\xi(v(p))=0$ for all $v\in \fg$ if $(p,\xi)\in T^*_G \C^n|_p$ and $\epsilon(p)=0$ if and only if $p=0$ since the group $G$ acts with no fixed vectors.
			Therefore $\text{Char}(\sigma^\epsilon)\cap T^*_G \C^n=\{0\}\times \{0\}$ and $\sigma^\epsilon$ is a $G$-transversally elliptic symbol.

			\begin{definition}
			\label{expansions}
				Let $t\in G$ and $\alpha\in \fg^*$, then
				\begin{equation*}
					\left(\frac{1}{1-t^{-\alpha}}\right)^+=-\sum_{k=1}^\infty t^{k\alpha}\text{ and }\left(\frac{1}{1-t^{-\alpha}}\right)^-=\sum_{k=0}^\infty t^{-k\alpha},
				\end{equation*}
				are the expansions in positive and negative powers of $\alpha$. 
			Alternatively, $(\cdot)^\pm$ are the Laurent expansions of a rational function around $t=\infty$ and $t=0$, respectively.
			\end{definition}

			\begin{theorem}
			\label{indexweights}
			Let $\sigma^\epsilon$ be Atiyah's pushed symbol, then
			\begin{equation*}
			    \ind_G^{\C^n} (\sigma^\epsilon) (t)=\prod_{j=1}^n\left(\frac{1}{1-t^{\alpha_j}}\right)^{s_j}\in \hat{R}(G),
			\end{equation*}
			where $s_j=+$ if $\alpha_j(\epsilon)>0$ and $s_j=-$ if $\alpha_j(\epsilon)<0$.
			\end{theorem}
			\begin{proof}
				See \cite{atiyah_elliptic_1974} (Theorem 8.1) or \cite{berline_indice_1996}.
			\end{proof}

			Atiyah shows in \cite{atiyah_elliptic_1974} how to extend the ideas above to evaluate the index for symbol classes in $K_G(T_G^* M)$, when $M$ is a compact manifold with an action of a $(n+1)$-dimensional torus $G$.
			The $G$-action provides $M$ with a filtration by closed subsets
			\begin{equation*}
			    M=M_0\supset M_1\supset\cdots\supset M_{n+1}\supset M_{n+2}=\emptyset
			\end{equation*}
			where $M_i=\left\{p\in M \ | \ \dim G_p\geq i \right\}$.		This filtration determines a split exact sequence for each $i$,
			\begin{equation*}
			\begin{tikzcd}[column sep=small]
			    0\arrow{r} & K_G(T^*_G(M-M_i))\arrow{r}{} & K_G(T^*_G (M-M_{i+1}))\arrow{r}& K_G(T^*_G M|_{(M_i-M_{i+1})})\arrow{r}\arrow[bend right=33,swap]{l}{\theta_i}  & 0,
			\end{tikzcd}
			\end{equation*}
			and a decomposition
			\begin{equation*}
			    K_G(T^*_G M) \cong \bigoplus_{i=0}^{n+1} \theta_i K_G(T^*_G M|_ {(M_i-M_{i+1})}).
			\end{equation*}

			Note that $T^*_G M|_ {(M_i-M_{i+1})}$ is a complex vector bundle over $T^*_G(M_i-M_{i+1})$, therefore we can compose the splittings $\theta_i$ with the Thom isomorphism, which we denote by  $\phi_i $, and obtain
			\begin{equation*}
			    K_G(T^*_G M) \cong \bigoplus_{i=0}^{n+1} \phi_i K_G(T^*_G (M_i-M_{i+1})).
			\end{equation*}

			This decomposition allows us break up a symbol and evaluate its index on each piece separately. For instance, if $[\sigma]\in K_G(T^*_G M)$ then
			\begin{equation*}
			    \ind^M_G (\sigma)=\sum_{i=0}^{n+1} \ind^M_G(\phi_i (\sigma|_{M_i-M_{i+1}})).
			\end{equation*}

			To compute the index, we need to describe the maps $\phi_i$. The Thom isomorphism is well-understood. Let us recall the definition of the splitting maps $\theta_i$ from \cite{atiyah_elliptic_1974}. The approach is analogous to the construction of the pushed symbol in Theorem \ref{indexweights}.

			First, note that we can define $\theta_i$ on each connected component $K$ of $M_i-M_{i+1}$ separately, and each of these components is an open subset of a unique fixed point set $M^{\T^i}$, where $\T^i$ is a torus of dimension $i$. The normal bundle $N$ of $K$ has an action of $\T^i$ that leaves no non-zero vectors fixed. We can therefore pick a vector $\epsilon\in \Lie(\T^i)$ in a tubular neighbourhood $U\subset M$ of $K$ and proceed as in Theorem \ref{indexweights} to define the maps $\theta_i$.  Note that the splittings $\theta_i$ depend on the choice of vector fields used for the deformation.
			To use this decomposition, we need to evaluate the index on each level of the filtration. For a given symbol, the following proposition  gives a systematic approach for applying the localisation argument.

			\begin{proposition}
			\label{localization}
			Let $[\sigma]\in K_G(T_G^* M)$ and let $v$ be a vector field on $M$ such that:
			\begin{itemize}
			    \item $\sigma(x,\xi+\lambda \tilde{v}(p))$ is an isomorphism for every $p\in M\setminus M_{j}$, $\lambda\in \R\setminus \{0\}$, $\xi\in(T_G^* M)_x$ and some $j\in \{1,\ldots,n+1\}$;
			    \item $v(p)$ is tangent to the orbit $G\cdot p$;
			    \item $v(p)=0$ if and only if $p\in M_{j}$;
			\end{itemize}
			where $\tilde{v}=h(v,\cdot)$ for a $G$-invariant Riemannian metric $h$ on $M$.
			Then for $i<j$, we can use $v(p)$ to construct the map $\theta_i$ and we have $\theta_i(\sigma)=0$.
			\end{proposition}
			\begin{proof}

			If $i<j$, the vector field $v$ is  non-vanishing  on $M_i-M_{i+1}$. Using this vector field to define the splitting $\theta_i$, we have
			\begin{equation*}
			    \theta_i(\sigma)(p,\xi)=\sigma(p,\xi+g(\mid \xi \mid)\tilde{v}(p)),
			\end{equation*}
			where $(p,\xi)\in (T_G^* U)_p$ and $U$ is a tubular neighbourhood of $M_i-M_{i+1}$ in $M-M_{i+1}$ and $g$ is a bump function supported on a small neighbourhood of the zero section of $T_G^* U$.
			Since $\sigma(p,\xi)$ is an isomorphism for $\xi\neq 0$ it follows that $\theta_i(\sigma)(p,\xi)$ fails to be an isomorphism exactly when $\xi+g(\mid \xi \mid)\tilde{v}(p)=0$. Since $v(p)$ is tangent to $G\cdot p$ and $\xi$ is orthogonal to $G\cdot p$ we have
			\begin{equation*}
			    \xi+g(\mid \xi \mid)\tilde{v}(p)=0\iff \xi=0\text{ and } v(p)=0.
			\end{equation*}
			If $p\in M_i-M_{i+1}$, then $v(p)\neq 0$ and therefore $\theta_i(\sigma)(p,\xi)$ is an isomorphism for all $(p,\xi)\in (T_G^* U)_p$ and $\theta_i(\sigma)$ is a representative of the zero class in $K_G(T_G^* (M-M_{i+1}))$.  
			\end{proof}

			\begin{remark}
				Proposition \ref{localization} is particularly useful in the following situation.
				Suppose that there is a $k\in \{1,\ldots,n+1\}$ such that $M_j=\emptyset$ for all $j>k$. If there is a vector field $v(p)$ satisfying the above hypothesis for $j=k$, then only the level $k$ of the filtration contributes to the index and we have
				\begin{equation*}
					\ind^M_G(\sigma)=\ind^M_G\phi_{k}(\sigma|_{M_{k}}).
				\end{equation*}
			\end{remark}
\section{The index of \texorpdfstring{$\dolbeault_H^E$}{}}
	\label{index_dolbeault}
	Adopting the localisation technique outlined in Section \ref{localisation_section}, we  will compute the index of the twisted Dolbeault operator $\dolbeault_H^E$ and  show that the contributions to the index come from a finite number of closed Reeb orbits.

	Let $M$ be a $(2n+1)$-dimensional toric contact manifold of Reeb type, $n>1$,  equipped with an invariant Sasakian structure. We  proceed by constructing a vector field satisfying the hypothesis of Proposition \ref{localization} for $j=n$.
		Let us fix a determinant $\det\in\bigwedge^{n+1}\fg^*$ and a vector $R\in\fg$ that generates the Reeb vector field. Given a closed Reeb orbit $L$ corresponding to an edge of the moment cone, it follows by the good cone condition that there exists a vector $v_0^L\in \Z_G$ such that
		\begin{equation*}
		    \det(v_0^L,v_1^L,\ldots,v_n^L)=1,
		\end{equation*}
		where $v_i^L$, $i=1,\dots,n$, are the  cone normals at $L$. 
		Since $\{v_0^L,v_1^L,\ldots,v_n^L\}$ is an integral basis of $\Z_G$, its dual basis $\{\mu_L,w^1_L,\ldots,w^n_L\}$ is an integral basis of the integral weights lattice $\Z_G^*$.
	We expand
		\begin{equation*}
		    R=\mu_L(R)v_0^L+w_L^1(R)v_1^L+\cdots+w_L^n(R)v_n^L.
		\end{equation*}
		Since $R$ generates the Reeb, its infinitesimal action is non-zero everywhere. Therefore its $v_0^L$ component, $\mu_L(R)$, cannot be zero, since the infinitesimal action of $v_i^L$, $i=1,\ldots,n$, on $L$ is zero.
		It follows that
		\begin{equation*}
		    \det\left(R,v_1^L,\ldots,v_n^L\right)=\mu_L(R)\neq 0,
		\end{equation*}
		 and $\{R,v_1^L,\ldots,v_n^L\}$ forms a basis of $\fg$ (not necessarily integral).

		Given a vector $\epsilon\in \fg$, we will define the vector field $\epsilon^\perp$ as its orthogonal complement with respect to the Reeb vector field. More precisely, let $\epsilon\in\fg$ and $L_i$ be a Reeb orbit corresponding to an edge of the moment cone. Write
		\begin{equation*}
		    \epsilon=\eta_{L_i}(\epsilon)R+\eta_{L_i}^1(\epsilon)v_1^{L_i}+\cdots+\eta_{L_i}^n(\epsilon)v_n^{L_i}.
		\end{equation*}
		Let $U_i$ be an open neighbourhood of $L_i$ and $V_i$ a closed neighbourhood such that $L_i\subset V_i\subset U_i$. We can assume that the $U_i$'s are all disjoint.
		Define
		\begin{equation*}
		    \phi^{L_i}(p)=
		    \begin{cases}
		        0 & \text{if $p\in M\setminus U_i$},\\
		        -\eta_{L_i}(\epsilon) & \text{if $p\in V_i$},
		    \end{cases}
		\end{equation*}
		and extend it to smoothly interpolate between $0$ and $-\eta_{L_i}(\epsilon)$ on $U_i\setminus V_i$, defining a smooth bump function.
		Projecting out the Reeb vector field component corresponds to shrinking the contribution of $R$ to the vector field generated by $\epsilon$. We define $\epsilon^\perp$ as
		\begin{equation*}
		    \epsilon^\perp(p)=\epsilon(p)+\sum_{i=1}^N \phi^{L_i}(p)R(p).
		\end{equation*}

		\begin{definition}
		\label{polarizingvector}
		An element $\epsilon\in \fg$ is called a {\em polarizing vector} if $\eta^i_{L}(\epsilon)\neq 0$ for $i=1,\ldots,n$ and for every $L\subset M_n$. We say that a vector field $\epsilon^\perp$ is a {\em good deformation vector field} for $[\sigma]\in K_G(T_G^* M)$ if it satisfies the hypothesis of Proposition \ref{localization} for $j=n$.
		\end{definition}

		\begin{proposition}
		\label{polarization}
		If $\epsilon\in\fg$ is a polarizing vector, then $\epsilon^\perp$ is a good deformation vector field for $\sigma(\dolbeault_H^E)\in K_G(T_G^* M)$.
		\end{proposition}

		\begin{proof}
		 By construction  $\epsilon^\perp(p)$ is tangent to the $G$-orbits. The symbol $\sigma(\dolbeault_H)(p,v(p))$ is invertible if $v(p)$ is not parallel to the Reeb vector field $R(p)$, so we only need to prove that $\epsilon^\perp(p)=0$ if and only if $p\in M_n$.
		Let $p\in U_i$ and suppose that $\epsilon^\perp(p)=0$. Let $L=L_i$ be the orbit corresponding to the normals $v_1^L,\ldots,v_n^L$ and write
		\begin{equation*}
		    \epsilon^\perp(p)=(\eta_L(\epsilon)+\sum_{j=1}^N \phi^{L_j}(p))R(p)+\eta_L^1(\epsilon)v_1^L(p)+\cdots+\eta_L^n(\epsilon)v_n^L(p)=0.
		\end{equation*}
		This implies that
		\begin{equation*}
		    (\eta_L(\epsilon)+\sum_{j=1}^N \phi^{L_j}(p))R+\eta_L^1(\epsilon)v_1^L+\cdots+\eta_L^n(\epsilon)v_n^L\in \fg_p.
		\end{equation*}
		For $p\in U_i$, there is a finite number of possible isotropy algebras $\fg_x$; they are all generated by a subset of $\{v_1^L,\ldots,v_n^L\}$.
		This implies that 
		\begin{equation*}
		    (\eta_L(\epsilon)+\sum_{j=1}^N \phi^{L_j}(p))=0,
		\end{equation*}
		and since $\{v_1^L,\ldots,v_n^L\}$ is a linearly independent set, we have $
		    v_j^L\in \fg_x$  for all  $j=1,\ldots,n.$
		It follows that the image of $p$ under the moment map lies in the intersection of the faces determined by the normals $v_1^L,\ldots,v_n^L$, so $p$ is a point in the orbit $L$.
		Let $p\in M\setminus \bigcup_{j=1}^N U_j$, 
		we have 
		\begin{equation*}
		    \epsilon^\perp(p)=\epsilon(p)+\sum_{i=1}^N \cancelto{0}{\phi^{L_i}(p)}R(p)=\epsilon(p).
		\end{equation*}
		Let $L$ be one of the orbits $L_i$ and write
		\begin{equation*}
		    \epsilon=\eta_{L}(\epsilon)R+\eta_{L}^1(\epsilon)v_1^{L}+\cdots+\eta_{L}^n(\epsilon)v_n^{L}.
		\end{equation*}
		Since ${R,v_1^L,\ldots,v_n^L}$ is a basis of $\fg$ and $G$ acts freely on $M\setminus \bigcup_{j=1}^N U_j$, we have $\eta^j_L(\epsilon)\neq 0$ and $\epsilon(p)$ must be non-zero. Hence, $\epsilon^\perp$ is a good deformation vector field for $\sigma(\dolbeault_H)$ and also for $\sigma(\dolbeault_H^E)$ as they have the same characteristic sets.
		\end{proof}

		The level $n$ filtration $M_n \subset M$ is a disjoint union of closed Reeb orbits $L_e$ indexed by the set $E(C)$ of edges of the moment cone $C$, 
		\begin{equation*}
			M_n=\bigsqcup_{e\in E(C)} L_e.
		\end{equation*}
		Since $M_{n+1}=\emptyset$, it follows by Proposition \ref{polarization} that
		\begin{equation*}
		    \ind^M_G(\sigma(\dolbeault_H))=\ind^M_G\phi_n (\sigma(\dolbeault_H)|_{M_n})=\sum_{e\in E(C)} \ind^M_G \phi_n(\sigma(\dolbeault_H)|_{L_e}).
		\end{equation*}
			Given an orbit $L\subset M_n$ and a vector $\epsilon\in \fg$, we denote by $\epsilon_L^\perp$ the vector
			\begin{equation*}
			    \eta_{L}^1(\epsilon)v_1^{L}+\cdots+\eta_{L}^n(\epsilon)v_n^{L}\in \fg.
			\end{equation*}
		In a neighbourhood of each closed orbit $L\subset M_n$, there is a vector $\epsilon_L^\perp\in \fg$ generating the vector field $\epsilon^\perp$.

		\begin{proposition}
			\label{index_localized}
			Let $L\subset M_n$ be a closed Reeb orbit. For any $t\in G$, we have
			\begin{equation*}
	    		\ind^M_G\phi_n(\sigma(\dolbeault_H)|_L)(t)=\left(\frac{1}{1-t^{-{w}^1_L}}\right)^{s_L^1}\cdots\left(\frac{1}{1-t^{-{w}_L^n}}\right)^{s_L^1}\delta(1-t^{\mu_L}),
			\end{equation*}
			where $\{\mu_L,w^1_L,\ldots,w_L^n\}$ is a basis of the weight lattice $\Z_G^*$ dual to $\{v_0^L,\ldots,v_n^L\}$, $s_L^i=+$ if $w^i_L(\epsilon_L^\perp)>0$ and $s_L^i=-$ if $w^i_L(\epsilon_L^\perp)<0$.
		\end{proposition}
	    \begin{proof}

	        To evaluate the index, we need to understand the symbols $\phi_n(\sigma(\dolbeault_H|)_{L_e})$. The map $\phi_n$ is a composition of the Thom isomorphism with the splitting homomorphism $\theta_n$.
	        Let $L$ be a connected component of $M_n$ and let $N$ be its normal bundle in $M$. Since $L\subset M$ is an embedded circle, $N$ is a trivial complex bundle $N\cong L\times \C^n$. Write $G=G_L\times S_0^1$, where $G_L$ is the isotropy group associated with $L$ and $S_0^1$ the circle generated by the vector $v_0^L\in \Z_G$. Taking $G_1=G_L$, $G_2=S_0^1$, $M_1=L$ and $M_2=\C^n$ in \eqref{kmultiplication}, we get
	        \begin{equation*}
	            \boxtimes\colon K_{G}(T_{G}^* L)\otimes K_{G_{L}}(T_{G_L}^* \C^n)\to K_{G_{L}\times S_0^1}(T^*_{G_{L}\times S_0^1} (L\times \C^n))\cong K_G(T_G^* N).
	        \end{equation*}
	        The map $\phi_n$ is given by taking the Bott element $[\sigma^{\epsilon_L^\perp}]\in K_{G_L}(T^*_{G_L} \C^n)$ in this product. That is, given $\sigma\in K_G(T_G^*L)$ we have
	        \begin{equation*}
	        	\phi_n(\sigma)=\sigma\boxtimes \sigma^{\epsilon^\perp_L}\in K_G(T_G^* N).
	        \end{equation*}
	        Identifying $N$ with a tubular neighbourhood $U$ of $L$ in $M$ and using excision to extend the symbol to $M$, we obtain $\phi_n(\sigma)=\sigma\boxtimes \sigma^{\epsilon^\perp_L}\in K_G(T_G^* M)$.
	        It follows from Theorem \ref{multiplicative_property} that
	        \begin{equation*}
	        	\ind^M_G \phi_n(\dolbeault_H|_{L})=\ind_{G_L}^{\C^n}(\sigma^{\epsilon^\perp_L})\ind_{G}^L(0),
	        \end{equation*}
	        where $0$ is the zero operator on $L\cong S^1$  discussed in Example \ref{zero_operator}. According to Theorem \ref{indexweights},
	        \begin{equation*}
	        	\ind_{G_L}^{\C^n}(\sigma^{\epsilon^\perp_L})(g)=\left(\frac{1}{1-g^{-{\alpha}_1}}\right)^{s^1_L}\cdots\left(\frac{1}{1-g^{-{\alpha}_n}}\right)^{s^n_L},
	        \end{equation*}
	        where $\alpha_1,\ldots,\alpha_n$ are the weights of the $G_L$-action on $\C^n$, $g\in G_L$, $s_L^i=+$ if $\alpha^i_L(\epsilon_L^\perp)>0$ and $s_L^i=-$ if $\alpha^i_L(\epsilon_L^\perp)<0$. By Corollary  \ref{local_toric_contact}, the weights $\alpha_1,\ldots,\alpha_n$ determine a basis of the weight lattice $\Z_G^*\subset \fg^*$ that is dual to $\{v_1^L,\ldots,v_n^L\}\subset\fg_L$.

	        Writing $G=G_L\times S_0^1$, the $G$-action on $L$ is given by 
	        \begin{equation*}
	        	t=(g,s)\cdot p = sp,
	        \end{equation*}
	        where $t=(g,s)\in G=G_L\times S_0^1$ and $p\in L$.
	        We identify the subgroup $S^1_0\subset G$ generated by $v_0^L$ with $S^1$ via
	        \begin{equation*}
	        	e^{2\pi i s v_0^L}\mapsto e^{2\pi i s}\in S^1.
	        \end{equation*}
	        This identification  is determined by the weight $\mu_L\in \Z^*_G$ defined by $\mu_L(v_0^L)=1$ and $\mu_L(v_j^L)=0$, $j\neq 0$.
	        In fact, let $s=e^{2\pi i s v_0^L}\in S_0^1$. Then $
	        	s^{\mu_L}=e^{2\pi i s \mu_L(v_0^L)}=e^{2\pi i s}\in S^1$ and
	        therefore
	        \begin{equation*}
	        	\ind_G^L(0)(g)=\ind^L_{G_L\times S_0^1}(0)(t,s)=\ind^L_{S^1}(0)(s^{\mu_L}).
	        \end{equation*}
	        Since $S_0^1$ acts freely and transitively on $L$ we have
	        \begin{equation*}
	        	\ind^L_{S^1}(0)(s^{\mu_L})=\sum_{k=-\infty}^\infty s^{k{\mu_L}}=\delta(1-s^{\mu_L}).
	        \end{equation*}
	        We  extend the weight vectors $\alpha_i\in \Z_{G_L}^*$, $i=1,\ldots,n$ to $\Z_G^*$ by defining $\alpha_i(v_0^L)=0$. Denote these extensions by $w_L^i\in \Z_G^*$, $i=1,\ldots,n$.
	        Note that $\{\mu_L,w_L^1,\ldots,w_L^n\}$ is a basis of $\Z_G^*$ dual to $\{v_0^L,v_1^L,\ldots,v_n^L\}$ and $w_L^i(\epsilon_L^\perp)=\alpha_i(\epsilon_L^\perp)$.
	        Let $t=(g,s)\in G=G_L\times S_0^1$ and $\eta\in \fg^*$. We will write $t^\eta=(g,s)^\eta=g^\eta s^\eta.$ Since $G_L$ is generated by $\{v_1^L,\ldots,v_n^L\}$ and $S_0^1$ is generated by $v_0^L$, we have $t^{w_i}=(g,s)^{w_i}=g^{w_i}s^{w_i}=g^{\alpha_i}$ and $t^{\mu_L}=(g,s)^{\mu_L}=g^{\mu_L}s^{\mu_L}=s^{\mu_L}$.
	        Given $t=(g,s)\in G=G_L\times S_0^1$, we have
	        \begin{align*}
	        	\ind^M_G \phi_n(\sigma(\dolbeault_H)|_L)(t)&=\ind^M_{G_L\times S_0^1}(\sigma(\dolbeault_H)|_L)(g,s)\\
	        	&=\ind^{\C^n}_{G_L}(\dolbeault_{\epsilon^\perp})(g)\ind^L_{G_L\times S_0^1}(0)(g,s)\\
	        	&=\left(\frac{1}{1-t^{-{w}^1_L}}\right)^{s^1_L}\cdots\left(\frac{1}{1-t^{-{w}_L^n}}\right)^{s^n_L}\delta(1-t^{\mu_L}).
	        \end{align*}
		\end{proof}

		Next we allow for twistings by an auxiliary bundle and derive the main result of this section, which is a Lefschetz type formula for the index of $\dolbeault_H^E$. Let $L$ be a closed orbit corresponding to an edge of the moment cone $C$ and let $E\to M$ be a  $G$-equivariant transversally holomorphic bundle. Since $L=G/G_L$, the restriction $ E|_{L}$ to $L$ is a vector bundle of the form $G\times_{G_L} F_L$ for some $G_L$-module $F_L$. Since $G=L\times G_L$ we have
			\begin{equation*}
	   	    	G\times_{G_L} F=(L\times G_L)\times_{G_L} F=L\times F.
	   	    \end{equation*}
	Recall that $M_n$ is the disjoint union of closed Reeb orbits $L_e$ indexed by the edges of the moment cone $C$.
		\begin{theorem}
			\label{localizationformulaIndex}
			Let $\dolbeault_H^E$ be the horizontal Dolbeault operator on a  toric compact Sasaki manifold twisted by a $G$-equivariant transversally holomorphic bundle $E$.  For any $t\in G$,
			\begin{equation*}
				\ind^M_G \sigma(\dolbeault_H^E)(t)=\sum_{L\subset M_n} \prod_{i=1}^n\chi_{E|_{L}}(t)\left(\frac{1}{1-t^{-{w}^i_L}}\right)^{s^i_L} \delta(1-t^{\mu_L}),
			\end{equation*}
			where  $\chi_{E|_{L}}$ is the character of the $G_L$-module associated to the restriction $E|_{L}$, and
			$$	s_L^i=	    \begin{cases}
			        +&\text{ if } w^i_L(\epsilon_L^\perp)>0,\\
			        -&\text{ if } w^i_L(\epsilon_L^\perp) <0.
			    \end{cases}$$
		\end{theorem}
		\begin{proof}
			We have that
			\begin{equation*}
				\ind^M_G \sigma(\dolbeault_H^E)=\sum_{L\in M_n} \ind^M_G \phi_n(\sigma(\dolbeault_H^E)|_L).
			\end{equation*}
			Restricting $\sigma(\dolbeault_H^E)$ to $L\subset M_n$ we get the symbol $\sigma(\dolbeault_H)|_L\otimes E|_{L}$.  If follows that $E|_{L}=L\times F_L$, for some $G_L$-module $F_L$ and $\chi_{E|_{L}}=\chi_{F_L}$. 
			By the multiplicative property of the index,  Proposition \ref{trivial_twist} and Proposition \ref{index_localized}, we get
			\begin{align*}
				\ind^M_G \phi_n(\sigma(\dolbeault_H^E)|_L(t)&= \ind^M_G(\sigma(\dolbeault_H)|_L) \chi_{F_L}\\
				&=\ind^{\C^n}_{G_L}(\dolbeault_{\epsilon^\perp_L})(t)\ind^L_G(0)(t)\chi_{E|_{L}}(t)\\	
				&=\chi_{E|_{L}}(t) \left(\frac{1}{1-t^{-{w}^1_L}}\right)^{s^1_L}\cdots\left(\frac{1}{1-t^{-{w}_L^n}}\right)^{s^n_L}\delta(1-t^{\mu_L}).
			\end{align*}
			The result follows by summing over all the closed Reeb orbits in $M_n$ .
		\end{proof}

\section{A lattice point formula}
	\label{lattice}
In this section, we relate the index of the horizontal Dolbeault operator $\dolbeault_H$ to the lattice points of the moment cone.
	\subsection{Polar decomposition of polytopes}
			The Lawrence-Varchenko formula expresses the characteristic function of a polytope as an alternating sum of characteristic functions of certain cones associated to vertices of the polytope. By extending the formula to polyhedral rational cones, it will allow us to collect  the multiplicities in the expression for the index in Theorem \ref{localizationformulaIndex}, once expanded into power series.  
			
			We begin by presenting the formula and relating the characteristic function of the interior of a polytope to the dual cones.
		Let $P$ be a simple convex polytope in an $n$-dimensional vector space $V^*$.
		Let $F$ be a face of $P$. The tangent cone to $P$ at $F$ is defined by
		\begin{equation*}
		    C_F=\{y+r(x-y)|r\geq 0, y\in F,x\in P\}.
		\end{equation*}
		Let $\sigma_1,\ldots,\sigma_d$ denote the facets of $P$. Since $P$ is simple, exactly $n$ facets intersect at each vertex. We will denote the set of vertices of $P$ by $\text{Vert}(P)$. For each face $F$ of $P$, let 
		   $ I_F\subset \{1,\ldots,d\}$
		be the set of indices of the facets meeting at $F$ so that
		\begin{equation*}
		    i\in I_F\text{ if and only if }F\subset\sigma_i.
		\end{equation*}
		In particular, if $F=p\in\text{Vert}(P)$ we have 
	$	    i\in I_v\text{ if and only if } v\in\sigma_i.$
	
		Let $p\in\text{Vert}(P)$ and denote by $w^i_p$, $i\in I_p$, the edge vector emanating from $p$ that lies along the unique edge at $p$ which is not contained on the facet $\sigma_i$. Notice that the $w^i_p$ are only determined up to a positive scalar.

		\begin{definition}
			A vector $\xi\in V$ such that all the pairings $\left<w^i_p,\xi\right>$ are non-zero is called a {\em polarizing vector} for $P$.
		\end{definition}

		Let $H_1,\ldots, H_N$ be the hyperplanes in $V$ determined by the edges of $P$ under the pairing between $V$ and $V^*$. A vector $\xi\in V$ is a polarizing vector for $P$ if and only if it belongs to the complement
		\begin{equation*}
		    V_P=V\setminus(H_1\cup\cdots\cup H_N).
		\end{equation*}
		The connected components of $V_P$ are called chambers. The signs of the pairings $\left<w^i_p,\xi\right>$ depend only on the chamber of $V_P$ containing $\xi$.

		\begin{definition}
		Let $\xi\in V_P$ be a polarizing vector. For each vertex $p\in\text{Vert}(P)$ and each edge vector $w^i_p$ emanating from $p$, we define the corresponding polarized edge vector to be
		\begin{equation*}
		    {w^i_p}^\#=
		    \begin{cases}
		        w^i_p&\text{ if }\left<w^i_p,\xi\right>>0,\\
		        -w^i_p&\text{ if }\left<w^i_p,\xi\right><0.
		    \end{cases}
		\end{equation*}
		\end{definition}

		\begin{definition} Given a polarizing vector $\xi\in V_P$, the {\em polarized tangent cone} at  $p\in\text{Vert(P)}$ is defined by
		\begin{equation*}
		    C^\#_p=p+\sum_{w\in E_p^+(\xi)}\R_{< 0}w+\sum_{w\in E_p^-(\xi)}\R_{\geq 0}w,
		    	\end{equation*}
where
		\begin{equation*}
		    E_p^+(\xi)=\{w^i_p \ | \ \left<w^i_p,\xi\right>>0\}\text{ and } E_p^-(\xi)=\{w^i_p \ | \ \left<w^i_p,\xi\right><0\}.
		\end{equation*}
		\end{definition}

		\begin{theorem}[Lawrence-Varchenko]
			\label{LVtheorem}
			Let $P\subset V^*$ be a simple convex polytope and $\xi\in V_P$ a polarizing vector for $P$. Then for any $x\in V^*$, we have
			\begin{equation}
			\label{LV}
			    \mathbf{1}_{P}(x)=\sum_{p\in \text{Vert}(P)} (-1)^{\left|E_p^+(\xi)\right|} \mathbf{1}_{C^\#_p} (x),
			\end{equation}
			where $\mathbf{1}_{C^\#_p}$ is the characteristic function of the polarized  cone $C^\#_p$.
		\end{theorem}

		\begin{proof}
			See Theorem $3.2$ in \cite{karshon_exact_2007}.
		\end{proof}

		Next we show that by flipping the cones in \eqref{LV} yields a cone decomposition of the interior of the polytope $P$. 

		\begin{definition}
			Define the {\em dual polarized tangent cone} at $p\in \text{Vert}(P)$ by
			\begin{equation*}
			    \widecheck{C}^\#_p=p+\sum_{w\in E_p^+(\xi)}\R_{> 0}w+\sum_{w\in E_p^-(\xi)}\R_{\leq 0}w\\
			\end{equation*}
		\end{definition}
Suppose that $\xi\in V$ lies in one of the walls $H_j$ separating the chambers of $V_P$. Let $e$ be an edge of $P$ perpendicular to this wall and let $p$ be an endpoint of $e$. The edge vectors at $p$ are $w^j_p$ for $j\in I_e$, and an edge vector that lies along $e$ is denoted by $w^e_p$.
		\begin{definition}
			\label{dual_cone_edge}
			The dual polarized tangent cone at the edge $e$ is defined by
			\begin{equation*}
			    \widecheck{C}^\#_e=p+\R w^e_p+\sum_{w\in E_p^+(\xi)}\R_{> 0}w+\sum_{w\in E_p^-(\xi)}\R_{\leq 0}w.
			\end{equation*}		
		\end{definition}

		One  verifies that the cone $\widecheck{C}^\#_e$ is independent of the choice of endpoint of the edge $e$.
			We also note that if $x\in \widecheck{C}^\#_p$, then 
			\begin{equation}
			\label{innerprod}
			    \left<\xi,x\right>\geq \left<\xi,p\right>.
			\end{equation}
			Indeed, if $x\in \widecheck{C}^\#_p$ we have
			\begin{equation*}
			    x=p+ \sum_{w\in E_p^+(\xi)}a_w w+\sum_{w\in E_p^-(\xi)}b_w w,
			\end{equation*}
			where $a_w> 0$, $b_w\leq 0$. Therefore
			$$
			    \left<\xi,x\right>=\left<\xi,p\right>+ \sum_{w\in E_p^+(\xi)}\overbrace{a_w \left<\xi,w\right>}^{\geq 0}+\sum_{w\in E_p^-(\xi)}\overbrace{b_w \left<\xi,w\right>}^{\geq 0}\geq \left<\xi,p\right>
$$

		\begin{theorem}
			\label{interiorLVtheorem}
			Let $P\subset V^*$ be a simple convex polytope and $\xi\in V_P$ a polarizing vector for $P$. Then for any $x\in V^*$, we have
			\begin{equation}
				\label{interiorLVformula}
			    (-1)^n\mathbf{1}_{P^\circ}(x)=\sum_{p\in \text{Vert}(P)} (-1)^{\left|E_p^+(\xi)\right|} \mathbf{1}_{\widecheck{C}^\#_p} (x),
			\end{equation}
			where $n=\dim V$ and $P^\circ$ denotes the interior of $P$.
		\end{theorem}

		\begin{proof}
			The proof proceeds along the same lines as that of \autoref{LVtheorem} in \cite{karshon_exact_2007} and comes down to verifying the identity  \eqref{interiorLVformula} in three separate cases and proving independence of the choice of  polarizing vector $\xi$.
			\begin{itemize}
				\item[Case 1:] Suppose that $x\in P^\circ$. 			    Pick any polarizing vector $\xi\in V_P$. Let $p\in\text{Vert}(P)$ be the vertex for which $\left<\xi,p\right>$ is minimal. Then $E_p^-(\xi)=\emptyset$ and we have $P^\circ\subset {\widecheck{C}^\#_p}$. For any other vertex $q\in\text{Vert}(P)$, at least one of the $w^j_q$'s is flipped, and so ${\widecheck{C}^\#_q}\cap P^\circ=\emptyset$. Hence $P^\circ$ is disjoint from the cones $\widecheck{C}^\#_q$ for all other $q\neq p$ and \eqref{interiorLVformula}, when evaluated at $x$, reads $(-1)^n=(-1)^n$.

			    \item[Case 2:] Suppose that $x\in \partial P$. Let $\sigma$ be a facet that contains $x$ and $p\in\text{Vert}(P)$ be such that $p\in\sigma$. Assume that given another facet $\sigma'$, we have $x\notin \sigma'$ if $p\notin\sigma'$.
			    Choose a polarizing vector $\xi\in V_P$ such that 
			    \begin{equation*}
			        \left<\xi,p\right>= \min_{y\in P}\left<\xi,y\right>.
			    \end{equation*}
			    We have $E_p^-(\xi)=\emptyset$ and therefore $x\notin \widecheck{C}^\#_p$ because $\widecheck{C}^\#_p=C^\circ_p$ and $x\in\sigma\subset \partial C_p$.
			    We show that $x\notin \widecheck{C}^\#_q$ for any other $q\in\text{Vert}(P)$.
			    Suppose that $q-p$ is not an edge of $P$. Let $\sigma_q=\sigma\setminus \bigcup_{j\in I_q} \sigma_j$, then $\sigma_q\subset  C^\circ_q$. Since $E_q^-(\xi)\neq\emptyset$ we have that $\widecheck{C}^\#_q\cap C^\circ_q=\emptyset$ and therefore $\sigma_q\cap\widecheck{C}^\#_q=\emptyset$. Thus if $x\in\sigma_q$ for some $q\in\text{Vert}(P)$ we have $x\notin \widecheck{C}^\#_q$.
			    If $q-p$ is an edge of $P$, we have $\left<\xi,q-p\right>>0$. Any element $y$ of $\widecheck{C}^\#_q$ can be written uniquely as
			    \begin{equation}\label{lincomb}
			        y=q+a(p-q)+r,
			    \end{equation}
			    where $a\leq 0$ and $r$ is a linear combination of the edges $w^j_q$ that are not parallel to the edge $p-q$. Since $x\in\sigma$, we can write $x$ uniquely as
			    \begin{equation*}
			        x=q+b(p-q)+s,
			    \end{equation*}    
			    where $b\geq 0$ and $s$ is a linear combination of the edges $w^j_q$ that are not parallel to the edge $p-q$. Since $x$ does not belong to any facet that does not contain $p$, we have that $b\neq 0$, and it follows from \eqref{lincomb} that $x\notin \widecheck{C}^\#_q$. This proves that $x\notin \widecheck{C}^\#_q$ for any vertex $q\in \text{Vert}(P)$. Therefore \eqref{interiorLVformula}, when evaluated at $x$, reads $0=0$.

			    \item[Case 3:] Suppose that $x\notin P$.  
			    Choose a polarizing vector $\xi\in V_P$ satisfying
			    \begin{equation*}
			        \left<\xi,x\right> < \min_{y\in P}\left<\xi,y\right>.
			    \end{equation*}
			    It  follows from \eqref{innerprod} that $x$ is not in $\widecheck{C}^\#_p$ for any $p\in \text{Vert}(P)$. Thus \eqref{interiorLVformula} for the polarizing vector $\xi$, when evaluated at $x$, reads $0=0$.
		\end{itemize}

The final step is to show that the right-hand side of \eqref{interiorLVformula} is independent of $\xi$. 
			    More precisely, we prove that the right-hand side of \eqref{interiorLVformula} does not change when $\xi$ crosses the walls $H_j$.
			    Suppose $H_j$ is not perpendicular to any edge vectors at $p$. The signs of  $\left<\xi,w^j_p\right>$ do not change, so the cone $\widecheck{C}^\#_p$ does not change as $\xi$ crosses the wall. The vertices whose contributions to the right-hand side of \eqref{interiorLVformula} change as $\xi$ crosses $H_j$ come in pairs because each edge of $P$ that is perpendicular to $H_j$ has two endpoints.
			    For each such vertex $p$, denote by $Q_p(x)$ and $Q'_p(x)$ its contributions to the right-hand side of \eqref{interiorLVformula} before and after $\xi$ crossed $H_j$. Let $e$ be an edge perpendicular to $H_j$ and $p$ an endpoint of $e$. Let $Q_e(x)$ be the characteristic function of the cone $\widecheck{C}^\#_e$ corresponding to the value of $\xi$ as it crosses $H_j$. We have
			    \begin{equation*}
			        Q_p(x)=(-1)^{|E_p^+(\xi)|}\mathbf{1}_{{\widecheck{C}}^\#_p}\text{ and } Q'_p(x)=(-1)^{|E_p^+(\xi)|+1}\mathbf{1}_{{\widecheck{C'}}^\#_p},
			    \end{equation*}
			    therefore
			    \begin{align*}
			        Q_p(x)-S'_p(x)&=(-1)^{|E_p^+(\xi)|}\mathbf{1}_{{\widecheck{C}}^\#_p}-(-1)^{|E_{p}^+(\xi)|+1}\mathbf{1}_{{\widecheck{C'}}^\#_p}\\
			        &=(-1)^{|E_p^+(\xi)|}(\mathbf{1}_{{\widecheck{C}}^\#_p}+\mathbf{1}_{{\widecheck{C'}}^\#_p})\\
			        &=(-1)^{|E_p^+(\xi)|}(\mathbf{1}_{{\widecheck{C}}^\#_e})=(-1)^{|E_p^+(\xi)|}Q_e(x)
			    \end{align*}    
			    If $q$ is the other endpoint of $e$, then $|E_p^+(\xi)|=|E_{q}^+(\xi)|\pm 1$. Hence
			    \begin{equation*}
			        Q_{q}(x)-Q'_{q}(x)=(-1)^{|E_{q}^+(\xi)|}Q_e(x)=(-1)^{|E_p^+(\xi)|+1}Q_e(x)
			    \end{equation*}
			    and
		$$
			        (Q_p(x)+Q_{q}(x))-(Q'_p(x)+Q'_{q}(x))=(-1)^{|E_p^+(\xi)|}Q_e(x)+(-1)^{|E_p^+(\xi)|+1}Q_e(x)
			        =0
$$		
			    Thus crossing $H_j$ does not change the right-hand side of \eqref{interiorLVformula}.
		\end{proof}
		Formulas  \eqref{LV} and  \eqref{interiorLVformula}  also have an expression in terms of generating series.
		
		\begin{definition}
			\label{generating_series}
			Let $V$ be a vector space with basis $e_1,\ldots,e_n$ and $\Z_V$ its integral lattice. If $A\subset V^*$ is a subset, we denote the generating series of $A$ by
			\begin{equation*}
			    A(x)=\sum_{\mu\in A\cap \Z_V^*} x^\mu,
			\end{equation*}
			where $\Z_V^*$ is the dual of the integral lattice $\Z_V$, $x=(x_1,\ldots,x_n)$ and $x^\mu=(x_1^{\mu_1},\ldots,x_n^{\mu_n})$. 
		\end{definition}

		\begin{theorem}
			Let $P\subset V^*$ be a simple convex polytope and $\xi\in V_P$ a polarizing vector for $P$. Then,
			\begin{equation*}
			    P(x)=\sum_{p\in \text{Vert}(P)} (-1)^{\left|E_p^+(\xi)\right|} C^\#_p (x)
			\end{equation*}
		\end{theorem}
		\begin{proof}
			The proof follows directly from \autoref{LVtheorem}. We have
			\begin{align*}
			    \sum_{p\in \text{Vert}(P)} (-1)^{\left|E_p^+(\xi)\right|} C^\#_v (x)&=\sum_{p\in \text{Vert}(P)} (-1)^{\left|E_p^+(\xi)\right|} \sum_{\mu\in C^\#_v\cap \Z_V^*} x^\mu\\
			    &=\sum_{p\in \text{Vert}(P)} (-1)^{\left|E_p^+(\xi)\right|} \sum_{\mu\in \Z_V^*} \mathbf{1}_{C^\#_v}(\mu) x^\mu\\
			    &=\sum_{\mu\in \Z_V^*}\left(\sum_{p\in \text{Vert}(P)} (-1)^{\left|E_p^+(\xi)\right|}  \mathbf{1}_{C^\#_v}(\mu)\right) x^\mu\\
			    &=\sum_{\mu\in \Z_V^*}\mathbf{1}_{P}(\mu) x^\mu=\sum_{\mu\in P\cap\Z_V^*} x^\mu=P(x).\\
			\end{align*}
		\end{proof}
		Similarly, for the dual Lawrence-Varchenko formula, we have:
		\begin{theorem}
			\label{openLVgenerating}
			\begin{equation*}
			    (-1)^n P^\circ(x)=\sum_{p\in \text{Vert}(P)} (-1)^{\left|E_p^+(\xi)\right|} \widecheck{C}^\#_p (x).
			\end{equation*}
		\end{theorem}
	\subsection{Polar decomposition of cones}
		\label{LVconesSECTION}
		In this section we explain how to adapt the Lawrence-Varchenko formula \eqref{LV} to produce a polar decomposition of a rational polyhedral cone. More precisely, let $P\subset V^*$ be a simple polytope and let $C\subset V^*\times \R^*$ be the cone over $P$, i.e.
		\begin{equation*}
			C=\{r(\eta,1^*)\in V^*\times \R^*\ | \ \eta\in P, r\geq 0\}.
		\end{equation*}

		The cone $C$ is the lift of the left-hand side of \eqref{LV} from $V^*$ to $V^*\times \R^*$. Lifting the right-hand side of \eqref{LV}, we can expect to obtain a polar decomposition of $C$. We will see  that this is almost true; one must introduce an error term to obtain an identity.
	 
		Let $p\in P$ be a vertex of the polytope. We will denote by $\mu_p\in \Z_V^*\times \Z^*$ the primitive edge vector of $C$ going through $p$. Given a polarizing vector $\xi\in V_P$ for $P$, let $C_p^\#$ be the polarized tangent cone of $P$ at $p$ and define
		\begin{equation*}
			K_p^\#=C_p^\#+\R \mu_p.
		\end{equation*}

		\begin{definition}
			Let $S(x)$ be the function defined as 
			\begin{equation*}
			    S(x)=\sum_{p\in Vert(P)} (-1)^{\left|E_p^+(\xi)\right|} \mathbf{1}_{K_p^\#}(x).
			\end{equation*}
		\end{definition}	
		The function  $S(x)$ is the lift of the right-hand side of \eqref{LV}.
		Let $\mathcal H=\{(\eta,\frac{1}{2}^*)\in V^*\times \R^* \ | \ \eta\in V^*\}$ be the characteristic hyperplane and let
		\begin{equation*}
		\mathcal    H_\lambda=\{(\eta,\lambda^*)\in V^*\times \R^* \ | \ \eta\in V^*, \lambda^*\in \R^*\},
		\end{equation*}
		be its parallel shifts.
		The polytope $P$ is the intersection of $C$ with $\mathcal H$ and for $\lambda\geq 0$ the intersection of $\mathcal H_\lambda$ with $C$ will be denoted $P_\lambda$. The projection $\pi\colon V^*\times \R^*\to V^*$ identifies the hyperplanes $\mathcal H_\lambda$ with the vector space $V^*$ and the polytopes $P_\lambda$  with $\lambda P\subset V^*$. When $\lambda<0$, we define $P_\lambda$ as the intersection $P_\lambda=\mathcal H_\lambda\cap (-C)$. In this case, $\pi$ also identifies $P_\lambda$ with $\lambda P\subset V^*$. If $\xi\in V_P$ is a polarizing vector for $P$, then $\xi$ is a polarizing vector for every $P_\lambda$.

		Let $p\in P$ be a vertex, then $\lambda p$ is a vertex of the polytope $\lambda P $. The intersection of $\mathcal H_\lambda$ with $K_p^\#$ is equal to $C_{\lambda p}^\#$ for $\lambda\geq 0$, where $C_{\lambda p}^\#$ is the polarized polarized tangent cone cone at $\lambda p$. When $\lambda<0$ the intersection becomes $K_p^\#\cap \mathcal H_\lambda=\widecheck{C}_{\lambda p }^\#$, the dual polarized tangent cone at $\lambda p$.
		Therefore, restricting $S$ to $\mathcal H_\lambda$ we get
		\begin{align*}
			S|_{ \mathcal H_\lambda}(x)&=\sum_{p\in Vert(P)} (-1)^{\left|E_v^+(\xi)\right|} \mathbf{1}_{K_p^\#}|_{ \mathcal H_\lambda}(x)=\sum_{p\in Vert(P)} (-1)^{\left|E_v^+(\xi)\right|} \mathbf{1}_{K_p^\#\cap \mathcal H_\lambda}(x)\\
			&=\sum_{p\in Vert(P)} (-1)^{\left|E_v^+(\xi)\right|} \mathbf{1}_{C_{\lambda p}^\#}(x)=P_\lambda(x),
		\end{align*}
		if $\lambda\geq 0$. Similarly, if $\lambda<0$ we have
		\begin{align*}
			S|_{ \mathcal H_\lambda}(x)&=\sum_{p\in Vert(P)} (-1)^{\left|E_v^+(\xi)\right|} \mathbf{1}_{K_p^\#}|_{ \mathcal H_\lambda}(x)=\sum_{p\in Vert(P)} (-1)^{\left|E_v^+(\xi)\right|} \mathbf{1}_{K_p^\#\cap  \mathcal H_\lambda}(x)\\
			&=\sum_{p\in Vert(P)} (-1)^{\left|E_v^+(\xi)\right|} \mathbf{1}_{\widecheck{C}_{\lambda p}^\#}(x)=(-1)^n P^\circ_\lambda(x).
		\end{align*}
		It follows that
		\begin{equation}
			\label{LVcones}
			S(x)=\mathbf{1}_C(x)+(-1)^n\mathbf{1}_{-C^\circ}(x).
		\end{equation}
		In a similar manner for the dual polarized tangent cones, let
		\begin{equation*}
			\widecheck{K}_p^\#=\widecheck{C}_p^\#+\R \mu_p.
		\end{equation*}
		We have $\widecheck{K}_p^\#\cap  \mathcal  H_\lambda=\widecheck{C}_{\lambda p}^\#$ for $\lambda\geq 0$ and $\widecheck{K}_p^\#\cap \mathcal  H_\lambda=C_{\lambda p}^\#$.

		\begin{definition}
			Let $\widecheck{S}(x)$ be the function defined by
			\begin{equation*}
				\widecheck{S}(x)=\sum_{p\in Vert(P)} (-1)^{\left|E_p^+(\xi)\right|} \mathbf{1}_{\widecheck{K}_p^\#}(x),
			\end{equation*}
			This is the lift of the right-hand side of \eqref{interiorLVformula}.
		\end{definition}

		Restricting to the hyperplanes $\mathcal H_\lambda$ we get
		\begin{align*}
			\widecheck{S}|_{\mathcal H_\lambda}(x)&=\sum_{p\in Vert(P)} (-1)^{\left|E_v^+(\xi)\right|} \mathbf{1}_{\widecheck{K}_p^\#}|_{\mathcal H_\lambda}(x)=\sum_{p\in Vert(P)} (-1)^{\left|E_v^+(\xi)\right|} \mathbf{1}_{\widecheck{K}_p^\#\cap\mathcal  H_\lambda}(x)\\
			&=\sum_{p\in Vert(P)} (-1)^{\left|E_v^+(\xi)\right|} \mathbf{1}_{\widecheck{C}_{\lambda p }^\#}(x)=(-1)^n P^\circ_\lambda(x),
		\end{align*}
		when $\lambda\geq 0$, and

		\begin{align*}
			\widecheck{S}|_{\mathcal  H_\lambda}(x)&=\sum_{p\in Vert(P)} (-1)^{\left|E_v^+(\xi)\right|} \mathbf{1}_{\widecheck{K}_p^\#}|_{\mathcal H_\lambda}(x)=\sum_{p\in Vert(P)} (-1)^{\left|E_v^+(\xi)\right|} \mathbf{1}_{\widecheck{K}_p^\#\cap \mathcal H_\lambda}(x)\\
			&=\sum_{p\in Vert(P)} (-1)^{\left|E_v^+(\xi)\right|} \mathbf{1}_{C_{\lambda p }^\#}(x)=P_\lambda(x),
		\end{align*}
		when $\lambda<0$.
		Therefore
		\begin{equation}
			\label{LVconesdual}
			\widecheck{S}(x)=(-1)^n\mathbf{1}_{C^\circ}(x)+\mathbf{1}_{-C}(x).
		\end{equation}
		Formulas \eqref{LVcones} and \eqref{LVconesdual} can again be expressed in terms of generating series. 
		\begin{proposition}
			\label{LVconesGenerating}
			\begin{align*}
				\sum_{p\in Vert(P)}(-1)^{\left|E_p^+(\xi)\right|} K_p^\#(x)&=C(x)+(-1)^n (-C^\circ)(x)\\
				&=\sum_{\mu\in C\cap (\Z_V^*\times \R^*)} x^\mu+(-1)^n\sum_{\mu\in (-C^\circ)\cap (\Z_V^*\times \R^*)} x^\mu.
			\end{align*}
		\end{proposition}
		\begin{proof}
			The proof is a straightforward application of \eqref{LVcones}.
			\begin{align*}
				\sum_{p\in Vert(P)}(-1)^{\left|E_p^+(\xi)\right|} K_p^\#(x)&=\sum_{p\in Vert(P)}(-1)^{\left|E_p^+(\xi)\right|}\sum_{\mu\in K_p^\#\cap (\Z_V^*\times \R^*)} x^\mu\\
				&=\sum_{p\in Vert(P)}(-1)^{\left|E_p^+(\xi)\right|}\sum_{\mu\in(\Z_V^*\times \R^*)} \mathbf{1}_{K_p^\#}(\mu)x^\mu\\
				&=\sum_{\mu\in (\Z_V^*\times \R^*)}\left(\sum_{p\in Vert(P)}(-1)^{\left|E_p^+(\xi)\right|} \mathbf{1}_{K_p^\#}(\mu)\right)x^\mu\\
				&=\sum_{\mu\in (\Z_V^*\times \R^*)}\left(\mathbf{1}_C(\mu)+(-1)^n\mathbf{1}_{(-C^\circ)}(\mu)\right)x^\mu\\
				&=\sum_{\mu\in C\cap (\Z_V^*\times \R^*)} x^\mu+(-1)^n\sum_{\mu\in (-C^\circ)\cap (\Z_V^*\times \R^*)} x^\mu.
			\end{align*}
		\end{proof}
Similarly, we have:
		\begin{proposition}	\label{LVconesdual2}
			\begin{align*}
				\sum_{p\in Vert(P)}(-1)^{\left|E_p^+(\xi)\right|} \widecheck{K}_p^\#(x)&=(-1)^n C^\circ(x)+(-C)(x)\\
				&=(-1)^n\sum_{\mu\in  C^\circ\cap (\Z_V^*\times \R^*)} x^\mu+\sum_{\mu\in (-C)\cap (\Z_V^*\times \R^*)} x^\mu.
			\end{align*}
		\end{proposition}
These results can be slightly generalised as follows. Let $W^*$ be a vector space and let $P\subset W^*$ be a simple convex polytope sitting on a hyperplane
		\begin{equation*}
		\mathcal	H=\left\{\eta\in W^*\ | \ \left<\eta,R\right>=1\right\}
		\end{equation*}
		determined by a vector $R\in W$.
		Let
		\begin{equation*}
			C=\{r\eta\in W^* \ | \ \eta\in P, r\geq 0\}
		\end{equation*}
		be the cone over $P$ and for each vertex $p\in P$ denote by $\mu_p$ the primitive edge vector of $C$ going through $p$.
		Let $\xi\in H^*_P$ be a polarizing vector for $P$. As above, for each vertex $p\in P$ denote by $C_p^\#\subset H$ the polarized tangent cone of $P$ at $p$ and by $\widecheck{C}_p^\#\subset H$ the dual polarized tangent cone of $P$ at $p$. 
		\begin{definition}
			\label{cones_definitions}
			Define the cones $K_p^\#$ and $\widecheck{K}_p^\#$ by
			\begin{equation*}
				K_p^\#=C_p^\#+\R\mu_p\;\;\text{and}\;\; \widecheck{K}_p^\#=\widecheck{C}_p^\#+\R\mu_p.
			\end{equation*}
		\end{definition}
		Let $\{e_1,\ldots,e_{n+1}\}$ be a basis of $W$ such that $e_{n+1}=R$. Then $\{e_1^*,\ldots,e_n^*\}$ is a basis of $\mathcal H$ and we have a linear isomorphism
		\begin{align*}
			T\colon W^*&\to\mathcal  H\times \R^*\\
			e_i^*&\mapsto (e_i^*,0)\\
			e_{n+1}^*&\mapsto \left(0,1^*\right).
		\end{align*}
		The map $T$ takes $\mathcal  H$ to the hyperplane
		\begin{equation*}
			T(\mathcal H)=\left\{\left(\eta,1^*\right)\in \mathcal H\times \R^*\ | \ \eta\in \mathcal  H\right\}
		\end{equation*}
		and $P$ to a polytope $T(P)\subset T(\mathcal  H)$. 
		Restricting $T$ to $\mathcal H$, we get a linear automorphism $T\colon\mathcal  H\to\mathcal  H$ that we will also denote by $T$. Let $T^{-1}$ be the inverse of $T$ and $(T^{-1})^*$ its adjoint. Let $v\in \mathcal H^*$ and $\eta\in \mathcal H$, then
		\begin{equation*}
			\left<\eta,v\right>=\left<T(\eta),(T^{-1})^*(v)\right>.
		\end{equation*}
		Since the edges of $P$ are taken to the edges of $T(P)$, the vector $(T^{-1})^*(\xi)$ induces a polarization of $T(P)$ such that $T(C_p^\#)=C_{T(p)}^\#$ for every vertex $p\in P$. The identity \eqref{LVcones} implies that
		\begin{equation*}
			\sum_{p\in Vert(P)} (-1)^{\left|E_p^+(\xi)\right|} \mathbf{1}_{K_{T(p)}^\#}(x)=\mathbf{1}_{T(C)}(x)+(-1)^n\mathbf{1}_{(-{T(C)}^\circ)}(x).
		\end{equation*}
Let $S(x)=\sum_{p\in Vert(P)} (-1)^{\left|E_p^+(\xi)\right|} \mathbf{1}_{K_{p}^\#}(x).$
		Since $K_{T(p)}^\#=T(K_{p}^\#)$, we have $\mathbf{1}_{K_{p}^\#}(x)=\mathbf{1}_{T(K_{p}^\#)}(Tx)$ and therefore
		\begin{align*}
			S(x)&=\sum_{p\in Vert(P)} (-1)^{\left|E_p^+(\xi)\right|} \mathbf{1}_{K_{p}^\#}(x)=\sum_{p\in Vert(P)} (-1)^{\left|E_p^+(\xi)\right|} \mathbf{1}_{T(K_{p}^\#)}(Tx)\\
			&=\mathbf{1}_{T(C)}(Tx)+(-1)^n\mathbf{1}_{(-{T(C)}^\circ)}(Tx)\\
			&=\mathbf{1}_{C}(x)+(-1)^n\mathbf{1}_{(-C^\circ)}(x),
		\end{align*}
		In summary, we have:
		\begin{proposition}
			Let $W^*$ be a vector space and $P\subset W^*$ a simple convex polytope on a hyperplane $H$ determined by a vector $R\in W$. Then, given a polarizing vector $\xi\in H^*_P$, we have
			\begin{equation*}
				S(x)=\sum_{p\in Vert(P)} (-1)^{\left|E_p^+(\xi)\right|} \mathbf{1}_{K_{p}^\#}(x)=\mathbf{1}_{C}(x)+(-1)^n\mathbf{1}_{(-C^\circ)}(x).
			\end{equation*}
		\end{proposition}
		A similar argument applied to the dual polarized tangent cones $\widecheck{K}_p^\#=\widecheck{C}_p^\#+\R\mu_p$ gives:
		\begin{proposition}
			\begin{equation*}
				\widecheck{S}(x)=\sum_{p\in Vert(P)} (-1)^{\left|E_p^+(\xi)\right|} \mathbf{1}_{\widecheck{K}_p^\#}(x)=(-1)^n\mathbf{1}_{C^\circ}(x)+\mathbf{1}_{-C}(x).
			\end{equation*}
		\end{proposition}
		The  identities involving generating series   in Proposition \ref{LVconesGenerating} and  \ref{LVconesdual2} continue to hold, with the lattice $\Z_W^*$ in place of $\Z_V^*$.

	\subsection{A formula for \texorpdfstring{$\ind_G^M(\sigma(\dolbeault_H))$}{}}
		Applying the results in the previous section, we show next how to obtain explicitly the multiplicities $m(\mu)$ associated to the weights $\mu\in \fg^*$ appearing in the index
		\begin{equation*}
			\label{index_multiplicities}
			\ind_G^M(\sigma(\dolbeault_H))(t)=\sum_{\mu \in \Z_G^*} m(\mu)t^\mu.
		\end{equation*}
Let $R\in\fg$ be the generator of the Reeb vector field, $\mathcal H$ the characteristic hyperplane determined by $R$ and $C$ the moment cone. The polytope $P=\mathcal H\cap C$ is the image of the $\alpha$-moment map $\phi_\alpha$.
		Each vertex of $P$ corresponds to an edge of $C$, corresponding to a connected component $L$ of $M_n$. 

		Given a vertex $p\in P$, let $L\subset M_n$ be the closed Reeb orbit corresponding to $p$. Since $C$ is a good cone there is a vector $v_0^L\in \fg$ such that $\{v_0^L,v_1^l,\ldots,v_n^L\}$ is an integral basis of $\Z_G$, where $\{v_1^l,\ldots,v_n^L\}$ is the set of normals to faces meeting at $p$. Let $\{\mu_L,w^1_L,\ldots,w^n_L\}$ be the dual basis of $\{v_0^L,v_1^l,\ldots,v_n^L\}$, Theorem \ref{localizationformulaIndex} tells us that if $\epsilon\in\fg$ is a polarizing vector, as in Definition \ref{polarizingvector}, then the index $\ind_G^M(\sigma(\dolbeault_H))$ is given by
		\begin{equation*}
			\ind_G^M \sigma(\dolbeault_H)(t)=\sum_{L\subset M_n} \left(\frac{1}{1-t^{-{w}^1_L}}\right)^{s^1_L}\cdots\left(\frac{1}{1-t^{-{w}_L^n}}\right)^{s^n_L}\delta(1-t^{\mu_L}),
		\end{equation*}
		where $s^i_L=+$ if $(\epsilon^\perp_L)>0$ and $s^i_L=-$ if $(\epsilon^\perp_L)<0$.
	    	Define the index sets
	    	\begin{equation*}
	    		W_L^+(\epsilon^\perp_L)=\{i\in \{1,\ldots, n\} \ | \ w_L^i(\epsilon^\perp_L)> 0\}
	    	\end{equation*}
	    	and
	    	\begin{equation*}
	    		W_L^-(\epsilon^\perp_L)=\{i\in \{1,\ldots, n\} \ | \ w_L^i(\epsilon^\perp_L)< 0\}.
	    	\end{equation*}
	    We can write
	    \begin{equation*}
	    	\left(\frac{1}{1-t^{-{w}^1_L}}\right)^{s^1_L}\cdots\left(\frac{1}{1-t^{-{w}_L^n}}\right)^{s^n_L}\delta(1-t^{\mu_L})=(-1)^{\left| W^+_L(\epsilon^\perp_L)\right|} \sum_{\mu\in \Z_G^*\cap K_L^\#(\epsilon^\perp)}t^\mu,
	    \end{equation*}
	    since $\{\mu_L,w^1_L,\ldots,w^n_L\}$ is an integral basis of $\Z_G^*$, and $K_L^\#(\epsilon^\perp)$ is the cone defined by
	    \begin{equation*}
	    	K_L^\#(\epsilon^\perp)=\R \mu_L+\sum_{i\in W_L^+(\epsilon^\perp_L)} \R_{>0} w^i_L+\sum_{i\in W_L^-(\epsilon^\perp_L)} \R_{\leq 0} w^i_L.
	    \end{equation*}
	    Since $\{\mu_L,w^1_L,\ldots,w^n_L\}$ is the dual basis of $\{v_0^L,\ldots,v_n^L\}$, the cone $K_L^\#(\epsilon^\perp)$ can also be written as
	    \begin{equation*}
	    	K_L^\#(\epsilon^\perp)=\bigcap_{i\in W_L^+(\epsilon^\perp_L)}\{w\in\fg^*\ | \  w(v_i^L) > 0 \}\bigcap_{i\in W_L^-(\epsilon^\perp_L)}\{w\in\fg^*\ | \  w(v_i^L)\leq 0 \}.
	    \end{equation*}
	    The following lemma gives yet another description of the cone $K_L^\#(\epsilon^\perp)$.

	    \begin{lemma}
	    \label{conedescription}
	    	Let $\{\eta^1,\ldots,\eta^n\}\subset\fg^*$ be a set of vectors satisfying $\eta^i(v_j^L)=\delta_{ij}$, $i,j=1,\ldots,n$ and let $K$ be the cone
	    	\begin{equation*}
	    		K=\R \mu_L+\sum_{i\in W_L^+(\epsilon^\perp_L)} \R_{> 0} \eta^i+\sum_{i\in W_L^-(\epsilon^\perp_L)} \R_{\leq 0} \eta^i.
	    	\end{equation*}
	    	Then $K=K_L^\#(\epsilon^\perp)$.
	    \end{lemma}
	    \begin{proof}
	    	Let $w\in K$ and write
	    	\begin{equation*}
	    		w=r\mu_L+\sum_{i\in W_L^+(\epsilon^\perp_L)} a_i \eta_L^i+\sum_{i\in W_L^-(\epsilon^\perp_L)} b_i \eta_L^i,
	    	\end{equation*}
	    	where $r, a_i, b_i\in \R$, $a_i> 0$ and $b_i\leq 0$, $i=1,\ldots, n$.
	    	Computing $w(v_i^L)$ for $i=1,\ldots,n$ we get
	    	\begin{equation*}
	    		w(v_i^L)=a_i> 0,\text{ for } i\in W_L^+(\epsilon^\perp_L)\text{ and }w(v_i^L)=b_i\leq 0\text{ for }i\in W_L^-(\epsilon^\perp_L).
	    	\end{equation*}
	    	Therefore 
	    	\begin{equation*}
	    		K\subset\bigcap_{i\in W_L^+(\epsilon^\perp_L)}\{w\in\fg^*\ | \ w(v_i^L)> 0 \}\bigcap_{i\in W_L^-(\epsilon^\perp_L)}\{w\in\fg^*\ | \ w(v_i^L)\leq 0 \}=K_L^\#(\epsilon^\perp).
	    	\end{equation*}

	    	Let $w\in K_L^\#(\epsilon^\perp)$, since $\{\mu_L,\eta^i,\ldots,\eta^n\}$ forms a basis for $\fg^*$ we can write
	    	\begin{equation*}
	    		w=r\mu_L+\sum_{i\in W_L^+(\epsilon^\perp_L)} a_i \eta_L^i+\sum_{i\in W_L^-(\epsilon^\perp_L)} b_i \eta_L^i.
	    	\end{equation*}
	    	Since $w\in K_L^\#(\epsilon^\perp)$, computing $w(v_i^L)$, $i=1,\ldots,n$, we find that $a_i> 0$ and $b_i\leq 0$ which implies that $w\in K$.
	    	Therefore $K=K_L^\#(\epsilon^\perp)$.
	    \end{proof}

	    Let $\{\eta_L,\eta_L^1,\ldots,\eta_L^n\}$ be the dual basis of $\{R,v_1^L,\ldots,v_n^L\}$, then $\eta_L^i(v_j^L)=\delta_{ij}$ and Lemma \ref{conedescription} implies that
	    \begin{equation*}
	    	K^\#_L(\epsilon^\perp)=\R \mu_L+\sum_{i\in W_L^+(\epsilon^\perp_L)} \R_{> 0} \eta^i_L+\sum_{i\in W_L^-(\epsilon^\perp_L)} \R_{\leq 0} \eta^i_L.
	    \end{equation*}
The index sets $W_L^+(\epsilon^\perp_L)$ and $W_L^-(\epsilon^\perp_L)$ can also be expressed in terms of the   basis $\{\eta_L,\eta_L^1,\ldots,\eta_L^n\}$.
	    \begin{lemma}
	    \label{signcollection}
	    	\begin{equation*}
	    		W_L^+(\epsilon^\perp_L)=\{i\in \{1,\ldots, n\}\ | \ w_L^i(\epsilon^\perp_L)> 0\}=\{i\in \{1,\ldots, n\}\ | \  \eta_L^i(\epsilon)> 0\}
	    	\end{equation*}
	    	and
	    	\begin{equation*}
	    		W_L^-(\epsilon^\perp_L)=\{i\in \{1,\ldots, n\}\ | \  w_L^i(\epsilon^\perp_L)< 0\}=\{i\in \{1,\ldots, n\}\ | \  \eta_L^i(\epsilon)< 0\}.
	    	\end{equation*}
	    \end{lemma}
	    \begin{proof}
	    	Since the weights $\{\alpha_1,\ldots,\alpha_n\}$ of the $G_L$-action on $\C^n$ form a dual basis to $\{v_1^L,\ldots,v_n^L\}$ in $\fg^*_L$ and
	    	\begin{equation*}
	    		\epsilon^\perp_L=\eta_L^1(\epsilon)v_1^L+\cdots+\eta_L^n(\epsilon)v_n^L,
	    	\end{equation*}
	    	we have $\alpha_i(\epsilon^\perp_L)=w^i_L(\epsilon^\perp_L)=\eta_L^i(\epsilon)$, for $i=1,\ldots,n$.
	    \end{proof}

		\begin{theorem}
		\label{index_formula}
The index of the horizontal Dolbeault operator $\dolbeault_H$ is given by
			\begin{equation*}
			\label{index_formula1}
				\ind^M_G(\sigma(\dolbeault_H))(t)=(-1)^n\sum_{\mu\in C^\circ \cap \Z_G^*} t^\mu+\sum_{\mu\in (-C)\cap \Z_G^*} t^\mu.
			\end{equation*}
		\end{theorem}
		\begin{proof} We note that  $\{\eta_L^1,\ldots,\eta_L^n\}$ are primitive vectors determining the edge directions of $P$ at $p$, since $\eta^i_L(R)=0$ for $i=1,\ldots,n$ and $\eta_L^i(v_j^L)=\delta_{ij}$. Let $\epsilon\in\fg$ be a polarizing vector, as in Proposition \ref{polarizingvector}, that is, $\epsilon$ satisfies $\eta_L^i(\epsilon)\neq 0$ for $i=1,\ldots,n$ and for all $L\subset M_n$.
		A polarizing vector for the polytope $P$ is a vector in the dual vector space $\mathcal H^*$. The vector $\epsilon\in \fg$ determines a polarizing vector $\epsilon_H$ for the polytope $P$ by
		\begin{equation*}
			\epsilon_H(\eta)=\eta(\epsilon), \text{ for all $\eta\in \mathcal H$}.
		\end{equation*}
		Since the edge vectors $\{\eta_L^1,\ldots,\eta_L^n\}$ satisfy $\epsilon_H(\eta_L^i)=\eta_L^i(\epsilon)\neq 0$, $i=1,\ldots,n$, $\epsilon_H$ is a polarizing vector for the polytope $P$. Since $\eta_L^i(\epsilon)=\epsilon_H(\eta_L^i)$, Lemma \ref{signcollection} implies that
		\begin{equation*}
	    	W_L^-(\epsilon^\perp_L)=\{i\in \{1,\ldots, n\} \ | \  \eta_L^i(\epsilon)< 0\}=E_p^-(\epsilon_H)
	    \end{equation*}
	    and
	    \begin{equation*}
	    	W_L^+(\epsilon^\perp_L)=\{i\in \{1,\ldots, n\} \ | \ \eta_L^i(\epsilon)> 0\}=E_p^+(\epsilon_H),
	    \end{equation*}
	    where $E_p^-(\epsilon_H)$ and $E_p^+(\epsilon_H)$ correspond to the edges $\eta_L^i$ of $P$ such that $\eta_L^i(\epsilon_H)< 0$ and $\eta_L^i(\epsilon_H)> 0$, respectively.
	    The vector $\epsilon_H$ determines a cone $\widecheck{K}_p^\#$, as in Definition \ref{cones_definitions}. Since the vectors $\eta_L^i$ are the edge vectors of $P$ meeting at $p$, the identity $\widecheck{K}_p^\#=K^\#_L(\epsilon^\perp)$ holds.
	    It follows that
	    \begin{align*}
	    	\ind_G^M \sigma(\dolbeault_H)(t)&=\sum_{L\subset M_n} \left(\frac{1}{1-t^{-{w}^1_L}}\right)^{s^1_L}\cdots\left(\frac{1}{1-t^{-{w}_L^n}}\right)^{s^n_L}\delta(1-t^{\mu_L})\\
	    	&=\sum_{L\subset M_n}(-1)^{\left| W^+_L(\epsilon^\perp_L)\right|} \sum_{\mu\in \Z_G^*\cap K_L^\#(\epsilon^\perp)}t^\mu\\&=\sum_{p\in \text{Vert(P)}}(-1)^{\left| E^+_p(\epsilon_H)\right|} \sum_{\mu\in \Z_G^*\cap \widecheck{K}_p^\#}t^\mu 
	    \end{align*}
	    and the result follows by applying Proposition \ref{LVconesdual2}.
	\end{proof}

\end{document}